\documentclass[11pt, a4paper, reqno]{amsart}

\usepackage{a4wide}
\usepackage{enumerate}
\usepackage{amssymb}
\usepackage[all,cmtip]{xy}

\newcommand{\shin}[1]{}
\newcommand{\seo}[1]{}

\renewcommand{\shin}[1]{{\color{red}\textsf{(Shin said: #1)}}}
\renewcommand{\seo}[1]{{\color{blue}\textsf{(Seo said: #1)}}}

\makeatletter
\def\markboth#1#2{%
 \begingroup
   \@temptokena{{#1}{#2}}\xdef\@themark{\the\@temptokena}%
   \mark{\the\@temptokena}%
 \endgroup
 \if@nobreak\ifvmode\nobreak\fi\fi}
\def\thanks#1{\g@addto@macro\thankses{\thanks{#1}}}
\makeatother

\usepackage{color}
\usepackage{ifpdf}
\ifpdf 
  \usepackage[pdftex]{graphicx}
  \DeclareGraphicsExtensions{.pdf,.png,.mps}
  \usepackage{pgf}
  \usepackage{tikz}
\else 
  \usepackage{graphicx}
  \DeclareGraphicsExtensions{.eps,.bmp}
  \usepackage{pgf}
  \usepackage{tikz}
  \usepackage{pstricks}
\fi
\usepackage{epic,bez123}
\usepackage{wrapfig}

\usepackage{amsthm}

\theoremstyle{plain}
\newtheorem{thm}{Theorem}
\newtheorem{cor}[thm]{Corollary}

\newtheorem{lem}[thm]{Lemma}
\newtheorem{prop}[thm]{Proposition}

\theoremstyle{definition}
\newtheorem{defn}[thm]{Definition}

\theoremstyle{remark}
\newtheorem{rmk}{Remark}

\usepackage{amsmath}

\DeclareMathOperator{\rank}{rank}

\DeclareMathOperator{\ind}{I}
\DeclareMathOperator{\zc}{zc}

\newcommand{\abs}[1]{\left|#1\right|}
\newcommand{\set}[1]{\left\{#1\right\}}

\def\F{\mathcal{F}}
\def\S{\mathcal{S}}
\def\CM{\mathcal{CM}}

\def\msa{\overline{\mathcal{S}}^{A}_{\mathrm{nc}}}
\def\sa{\mathcal{S}^{A}_{\mathrm{nc}}}
\def\sb{\mathcal{S}^{B}_{\mathrm{nc}}}
\def\cycle{\mathrm{cycle}}

\title{Annular noncrossing permutations and minimal transitive factorizations}

\author{Jang Soo Kim}
\address[Jang Soo Kim]{School of
  Mathematics, University of Minnesota, Minneapolis, Minnesota 55455, USA}
\email{kimjs@math.umn.edu}

\author{Seunghyun Seo}
\address[Seunghyun Seo]{Department of Mathematics Education, Kangwon National University, Hyoja-dong, Chuncheon, 200-701, South Korea}
\email{shyunseo@kangwon.ac.kr}

\author{Heesung Shin}
\address[Heesung Shin]{Department of Mathematics, Inha University, 100 Inharo, Nam-gu, Incheon, 402-751, South Korea}
\email{shin@inha.ac.kr}


\date{\today}

\begin{document}
\maketitle
\begin{abstract}
  We give two combinatorial proofs of Goulden and Jackson's formula for the
  number of minimal transitive factorizations of a permutation when the
  permutation has two cycles. We use the recent result of Goulden, Nica, and
  Oancea on the number of maximal chains of annular noncrossing partitions of
  type $B$.
\end{abstract}

\section{Introduction}
Given an integer partition $\lambda = (\lambda_1, \lambda_2, \dots, \lambda_{\ell})$ of $n$, denote by $\alpha_{\lambda}$
the permutation
$$(1 \dots \lambda_1) (\lambda_1+1 \dots \lambda_1+\lambda_2) \dots (n-\lambda_{\ell} + 1 \dots n)$$
of the set $\set{1,2,\dots,n}$ in the cycle notation.
Let $\F_{\lambda}$ be the set of all $(n+\ell-2)$-tuples $(\eta_1, \dots, \eta_{n+\ell-2})$ of transpositions such that
\begin{enumerate}[(1)]
\item $\eta_1 \cdots \eta_{n+\ell-2} = \alpha_{\lambda}$ and
\item $\set{\eta_1, \dots, \eta_{n+\ell-2}}$ generates the symmetric group $\S_{n}$.
\end{enumerate}
Such tuples are called \emph{minimal transitive factorizations} of the permutation $\alpha_{\lambda}$ of type~$\lambda$,
which are related to the branched covers of the sphere suggested by Hurwitz \cite{Hur91, Str96}.

In 1997, using algebraic methods Goulden and Jackson \cite{GJ97} proved that
\begin{equation}\label{eq:gj-formula}
\abs{\F_{\lambda}} = (n+\ell-2)! \, n^{\ell-3} \prod_{i=1}^{\ell} \frac{\lambda_i^{\lambda_i}}{(\lambda_i-1)!}.
\end{equation}
Bousquet-M\'elou and Schaeffer \cite{B-M2000} proved a more general formula than
\eqref{eq:gj-formula} and obtained \eqref{eq:gj-formula} using the principle of
inclusion and exclusion. Irving \cite{Irv09} studied the enumeration of minimal transitive factorizations into cycles instead of transpositions.

If $\lambda=(n)$, the formula \eqref{eq:gj-formula} yields
\begin{eqnarray}
\abs{\F_{(n)}} = n^{n-2}, \label{eq:n-formula}
\end{eqnarray}
and there are several combinatorial proofs of \eqref{eq:n-formula} \cite{Bia00, GY02, Mos89}.

If $\lambda = (p,q)$, the formula \eqref{eq:gj-formula} yields
\begin{eqnarray}
\abs{\F_{(p,q)}} = \frac{pq}{p+q} {p+q \choose q} p^p q^q. \label{eq:pq-formula}
\end{eqnarray}
A few special cases of \eqref{eq:pq-formula} have bijective proofs:
by Kim and Seo \cite{KS03} for the case $(p,q)=(1,n-1)$, and by Rattan
\cite{Rat06} for the cases $(p,q)=(2,n-2)$ and $(p,q)=(3,n-3)$. There are no
simple combinatorial proofs for other $(p,q)$.

Recently, Goulden et al. \cite{GNO11} showed that the number of maximal chains
in the poset $NC^{(B)} (p,q)$ of annular noncrossing partitions of type $B$ is
\begin{eqnarray}
{p+q \choose q} p^p q^q + \sum_{c \ge 1} 2c {p+q \choose p-c} p^{p-c} q^{q+c}.
 \label{eq:nc-formula}
\end{eqnarray}
Interestingly it turns out that half the sum in \eqref{eq:nc-formula} is equal
to the number in \eqref{eq:pq-formula}:
\[
\sum_{c \ge 1} c {p+q \choose p-c} p^{p-c} q^{q+c}
= \frac{pq}{p+q} {p+q \choose q} p^p q^q.
\]

In this paper we will give two combinatorial proofs of \eqref{eq:pq-formula}
using the results in \cite{GNO11}. The rest of this paper is organized as
follows.  In Section~\ref{sec:signed-permutations} we recall the poset
$\sb(p,q)$ of annular noncrossing permutations of type $B$ which is isomorphic
to the poset $NC^{(B)}(p,q)$ of annular noncrossing partitions of type $B$, and
show that the number of connected maximal chains in $\sb(p,q)$ is equal to
$\frac{2pq}{p+q} {p+q \choose q} p^p q^q$.  In Section~\ref{sec:2-1-map} we
prove that there is a 2-1 map from the set of connected maximal chains in
$\sb(p,q)$ to $\F_{(p,q)}$, thus completing a combinatorial proof of
\eqref{eq:pq-formula}. In Section~\ref{sec:mark-annul-noncr} we give another
combinatorial proof of \eqref{eq:pq-formula} by introducing marked annular
noncrossing permutations of type $A$.

\section{Connected maximal chains}
\label{sec:signed-permutations}

A \emph{signed permutation} is a permutation $\sigma$ on $\set{\pm 1,\ldots,\pm
  n}$ satisfying $\sigma(-i)=-\sigma(i)$ for all $i \in \set{1,\ldots,n}$. We
denote by $B_n$ the set of signed permutations on $\set{\pm 1,\ldots,\pm n}$.

We will use the two notations
\begin{align*}
[a_1~ a_2 \dots a_k] &= (a_1~ a_2 \dots a_k~ -a_1~ -a_2 \ldots -a_k),\\
((a_1~a_2 \dots a_k)) &= (a_1~ a_2 \dots a_k)(-a_1~ -a_2 \ldots -a_k),
\end{align*}
and call $[a_1~ a_2 \dots a_k]$ a \emph{zero cycle} and $((a_1~a_2 \dots a_k))$
a \emph{paired nonzero cycle}.  We also call the cycles $\epsilon_i:=[i]=(i~-i)$
and $((i~j))$ \emph{type $B$ transpositions}, or simply transpositions if there
is no possibility of confusion.

For $\pi\in B_n$, the \emph{absolute length} $\ell(\pi)$ is defined to be the
smallest integer $k$ such that $\pi$ can be written as a product of $k$ type $B$
transpositions. The \emph{absolute order} on $B_n$ is defined by
\[
\pi\leq \sigma \quad \Leftrightarrow \quad \ell(\sigma) = \ell(\pi) +
\ell(\pi^{-1}\sigma).
\]

From now, we fix positive integers $p$ and $q$.
The poset $\sb(p,q)$ of \emph{annular noncrossing permutations of type $B$} is defined by
\[
\sb(p,q) := [\epsilon, \gamma_{p,q}] = \set{\sigma \in B_{p+q} : \epsilon \leq
  \sigma \leq \gamma_{p,q}} \subseteq B_{p+q},
\]
where $\epsilon$ is the identity in $B_{p+q}$ and $ \gamma_{p,q} =[1 \dots
p][p+1 \dots p+q]$. Figure~\ref{fig:hasse} shows the Hasse diagram for
$\sb(2,1)$. Then $\sb(p,q)$ is a graded poset with rank function
\begin{equation}
\label{eq:rankB}
\rank(\sigma) = (p+q) - (\text{$\#$ of paired nonzero cycles of $\sigma$}).
\end{equation}

\begin{figure}[t]
$$
\xy
\small
\xymatrixcolsep{-3px}
\xymatrix{
 &&&& \gamma_{2,1}=[1~2][3] \ar@{-}[lllld] \ar@{-}[llld] \ar@{-}[lld]
\ar@{-}[ld] \ar@{-}[d] \ar@{-}[rd] \ar@{-}[rrd] \ar@{-}[rrrd]
\ar@{-}[rrrrd] &&&& \\
 [1][3] \ar@{-}[rdd] \ar@{-}[rrdd] \ar@{-}[rrrdd] \ar@{-}[rrrrdd]
 & ((1~2))[3] \ar@{-}[ldd] \ar@{-}[rrrdd]
 & ((1~2~3)) \ar@{-}[lldd] \ar@{-}[dd] \ar@{-}[rrrdd]
 & ((1~2~{-3})) \ar@{-}[llldd] \ar@{-}[dd] \ar@{-}[rrrdd]
 & [1~2] \ar@{-}[lllldd] \ar@{-}[llldd] \ar@{-}[rrrdd] \ar@{-}[rrrrdd]
 & ((1~{-3}~{-2})) \ar@{-}[lldd] \ar@{-}[dd] \ar@{-}[rrrdd]
 & ((1~3~{-2})) \ar@{-}[lllldd] \ar@{-}[dd] \ar@{-}[rrdd]
 & ((1~{-2}))[3] \ar@{-}[llldd] \ar@{-}[rdd]
 & [2][3] \ar@{-}[lllldd] \ar@{-}[llldd] \ar@{-}[lldd] \ar@{-}[ldd]
 \\
 &&&&&&&&\\
 ((1~2)) & [1] & ((1~3)) & ((1~{-3})) & [3] & ((2~3)) & ((2~{-3})) &
[2] & ((1~{-2})) \\
 &&&& \epsilon \ar@{-}[llllu] \ar@{-}[lllu] \ar@{-}[llu] \ar@{-}[lu]
\ar@{-}[u] \ar@{-}[ru] \ar@{-}[rru] \ar@{-}[rrru] \ar@{-}[rrrru] &&&&
}
\endxy
$$
\caption{The Hasse diagram for $\sb(2,1)$.}
\label{fig:hasse}
\end{figure}

Nica and Oancea \cite{NO09} showed that $\sigma\in\sb(p,q)$ if and only if
$\sigma$ can be drawn without crossing arrows inside an annulus in which the
outer circle has integers $1,2,\dots,p,-1,-2,\dots,-p$ in clockwise order and
the inner circle has integers $p+1,p+2,\dots,p+q,-p-1,-p-2,\dots,-p-q$ in
counterclockwise order, see Figure~\ref{fig:annulus}. They also showed that
$\sb(p,q)$ is isomorphic to the poset $NC^{(B)}(p,q)$ of annular noncrossing
partitions of type $B$.

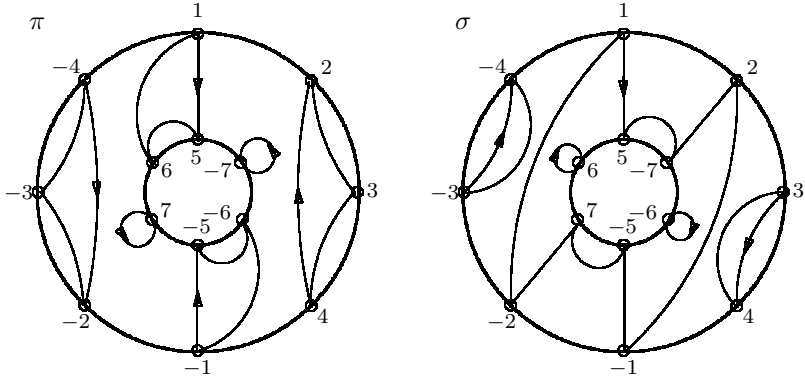
\begin{figure}[t]
\centering
\ifpdf
\begin{pgfpicture}{-5.64mm}{-5.71mm}{102.78mm}{48.43mm}
\pgfsetxvec{\pgfpoint{0.70mm}{0mm}}
\pgfsetyvec{\pgfpoint{0mm}{0.70mm}}
\color[rgb]{0,0,0}\pgfsetlinewidth{0.30mm}\pgfsetdash{}{0mm}
\pgfsetlinewidth{0.60mm}\pgfcircle[stroke]{\pgfxy(30.00,30.00)}{21.16mm}
\pgfcircle[stroke]{\pgfxy(30.00,30.00)}{7.00mm}
\pgfcircle[fill]{\pgfxy(30.00,60.00)}{0.70mm}
\pgfsetlinewidth{0.30mm}\pgfcircle[stroke]{\pgfxy(30.00,60.00)}{0.70mm}
\pgfcircle[fill]{\pgfxy(30.00,0.00)}{0.70mm}
\pgfcircle[stroke]{\pgfxy(30.00,0.00)}{0.70mm}
\pgfcircle[fill]{\pgfxy(60.00,30.00)}{0.70mm}
\pgfcircle[stroke]{\pgfxy(60.00,30.00)}{0.70mm}
\pgfcircle[fill]{\pgfxy(0.00,30.00)}{0.70mm}
\pgfcircle[stroke]{\pgfxy(0.00,30.00)}{0.70mm}
\pgfcircle[fill]{\pgfxy(8.77,51.31)}{0.70mm}
\pgfcircle[stroke]{\pgfxy(8.77,51.31)}{0.70mm}
\pgfcircle[fill]{\pgfxy(51.28,51.14)}{0.70mm}
\pgfcircle[stroke]{\pgfxy(51.28,51.14)}{0.70mm}
\pgfcircle[fill]{\pgfxy(51.28,8.43)}{0.70mm}
\pgfcircle[stroke]{\pgfxy(51.28,8.43)}{0.70mm}
\pgfcircle[fill]{\pgfxy(8.77,8.58)}{0.70mm}
\pgfcircle[stroke]{\pgfxy(8.77,8.58)}{0.70mm}
\pgfcircle[fill]{\pgfxy(30.00,40.00)}{0.70mm}
\pgfcircle[stroke]{\pgfxy(30.00,40.00)}{0.70mm}
\pgfcircle[fill]{\pgfxy(30.00,20.00)}{0.70mm}
\pgfcircle[stroke]{\pgfxy(30.00,20.00)}{0.70mm}
\pgfcircle[fill]{\pgfxy(38.06,35.60)}{0.70mm}
\pgfcircle[stroke]{\pgfxy(38.06,35.60)}{0.70mm}
\pgfcircle[fill]{\pgfxy(38.34,24.85)}{0.70mm}
\pgfcircle[stroke]{\pgfxy(38.34,24.85)}{0.70mm}
\pgfcircle[fill]{\pgfxy(21.37,24.85)}{0.70mm}
\pgfcircle[stroke]{\pgfxy(21.37,24.85)}{0.70mm}
\pgfcircle[fill]{\pgfxy(21.51,35.60)}{0.70mm}
\pgfcircle[stroke]{\pgfxy(21.51,35.60)}{0.70mm}
\pgfmoveto{\pgfxy(20.94,25.42)}\pgfcurveto{\pgfxy(17.74,28.20)}{\pgfxy(13.49,23.54)}{\pgfxy(16.56,20.61)}\pgfcurveto{\pgfxy(19.40,17.89)}{\pgfxy(23.66,21.76)}{\pgfxy(21.22,24.85)}\pgfstroke
\pgfmoveto{\pgfxy(8.77,51.31)}\pgfcurveto{\pgfxy(8.79,46.87)}{\pgfxy(7.87,42.49)}{\pgfxy(6.09,38.43)}\pgfcurveto{\pgfxy(4.74,35.38)}{\pgfxy(2.93,32.57)}{\pgfxy(0.71,30.09)}\pgfstroke
\pgfmoveto{\pgfxy(0.43,30.09)}\pgfcurveto{\pgfxy(2.21,27.86)}{\pgfxy(3.73,25.44)}{\pgfxy(4.95,22.87)}\pgfcurveto{\pgfxy(7.08,18.39)}{\pgfxy(8.29,13.53)}{\pgfxy(8.49,8.58)}\pgfstroke
\pgfmoveto{\pgfxy(8.77,51.45)}\pgfcurveto{\pgfxy(10.34,44.48)}{\pgfxy(11.14,37.37)}{\pgfxy(11.18,30.23)}\pgfcurveto{\pgfxy(11.22,23.09)}{\pgfxy(10.49,15.98)}{\pgfxy(9.00,9.00)}\pgfstroke
\pgfmoveto{\pgfxy(30.00,40.13)}\pgfcurveto{\pgfxy(29.00,43.05)}{\pgfxy(25.55,44.29)}{\pgfxy(22.92,42.68)}\pgfcurveto{\pgfxy(20.52,41.21)}{\pgfxy(19.93,37.97)}{\pgfxy(21.65,35.74)}\pgfstroke
\pgfmoveto{\pgfxy(21.37,35.46)}\pgfcurveto{\pgfxy(18.49,39.65)}{\pgfxy(17.76,44.94)}{\pgfxy(19.38,49.75)}\pgfcurveto{\pgfxy(21.07,54.74)}{\pgfxy(25.08,58.59)}{\pgfxy(30.14,60.08)}\pgfstroke
\pgfmoveto{\pgfxy(30.14,60.08)}\pgfcurveto{\pgfxy(30.14,60.08)}{\pgfxy(30.14,39.71)}{\pgfxy(30.14,39.71)}\pgfstroke
\pgfmoveto{\pgfxy(11.04,31.08)}\pgflineto{\pgfxy(11.18,29.10)}\pgfstroke
\pgfmoveto{\pgfxy(11.18,29.10)}\pgflineto{\pgfxy(11.68,31.94)}\pgflineto{\pgfxy(10.28,31.84)}\pgflineto{\pgfxy(11.18,29.10)}\pgfclosepath\pgffill
\pgfmoveto{\pgfxy(11.18,29.10)}\pgflineto{\pgfxy(11.68,31.94)}\pgflineto{\pgfxy(10.28,31.84)}\pgflineto{\pgfxy(11.18,29.10)}\pgfclosepath\pgfstroke
\pgfmoveto{\pgfxy(15.71,22.30)}\pgflineto{\pgfxy(16.84,20.75)}\pgfstroke
\pgfmoveto{\pgfxy(16.84,20.75)}\pgflineto{\pgfxy(15.76,23.42)}\pgflineto{\pgfxy(14.63,22.60)}\pgflineto{\pgfxy(16.84,20.75)}\pgfclosepath\pgffill
\pgfmoveto{\pgfxy(16.84,20.75)}\pgflineto{\pgfxy(15.76,23.42)}\pgflineto{\pgfxy(14.63,22.60)}\pgflineto{\pgfxy(16.84,20.75)}\pgfclosepath\pgfstroke
\pgfsetlinewidth{0.15mm}\pgfmoveto{\pgfxy(30.00,50.88)}\pgflineto{\pgfxy(30.00,48.62)}\pgfstroke
\pgfmoveto{\pgfxy(30.00,48.62)}\pgflineto{\pgfxy(30.70,51.42)}\pgflineto{\pgfxy(29.30,51.42)}\pgflineto{\pgfxy(30.00,48.62)}\pgfclosepath\pgffill
\pgfmoveto{\pgfxy(30.00,48.62)}\pgflineto{\pgfxy(30.70,51.42)}\pgflineto{\pgfxy(29.30,51.42)}\pgflineto{\pgfxy(30.00,48.62)}\pgfclosepath\pgfstroke
\pgfmoveto{\pgfxy(30.00,50.88)}\pgflineto{\pgfxy(30.00,48.62)}\pgfstroke
\pgfmoveto{\pgfxy(30.00,48.62)}\pgflineto{\pgfxy(30.70,51.42)}\pgflineto{\pgfxy(29.30,51.42)}\pgflineto{\pgfxy(30.00,48.62)}\pgfclosepath\pgffill
\pgfmoveto{\pgfxy(30.00,48.62)}\pgflineto{\pgfxy(30.70,51.42)}\pgflineto{\pgfxy(29.30,51.42)}\pgflineto{\pgfxy(30.00,48.62)}\pgfclosepath\pgfstroke
\pgfmoveto{\pgfxy(30.14,60.08)}\pgfcurveto{\pgfxy(30.14,60.08)}{\pgfxy(30.14,39.71)}{\pgfxy(30.14,39.71)}\pgfstroke
\pgfmoveto{\pgfxy(21.37,35.46)}\pgfcurveto{\pgfxy(18.49,39.65)}{\pgfxy(17.76,44.94)}{\pgfxy(19.38,49.75)}\pgfcurveto{\pgfxy(21.07,54.74)}{\pgfxy(25.08,58.59)}{\pgfxy(30.14,60.08)}\pgfstroke
\pgfmoveto{\pgfxy(30.00,40.13)}\pgfcurveto{\pgfxy(29.00,43.05)}{\pgfxy(25.55,44.29)}{\pgfxy(22.92,42.68)}\pgfcurveto{\pgfxy(20.52,41.21)}{\pgfxy(19.93,37.97)}{\pgfxy(21.65,35.74)}\pgfstroke
\pgfmoveto{\pgfxy(11.04,31.08)}\pgflineto{\pgfxy(11.18,29.10)}\pgfstroke
\pgfmoveto{\pgfxy(11.18,29.10)}\pgflineto{\pgfxy(11.68,31.94)}\pgflineto{\pgfxy(10.28,31.84)}\pgflineto{\pgfxy(11.18,29.10)}\pgfclosepath\pgffill
\pgfmoveto{\pgfxy(11.18,29.10)}\pgflineto{\pgfxy(11.68,31.94)}\pgflineto{\pgfxy(10.28,31.84)}\pgflineto{\pgfxy(11.18,29.10)}\pgfclosepath\pgfstroke
\pgfmoveto{\pgfxy(8.77,51.45)}\pgfcurveto{\pgfxy(10.34,44.48)}{\pgfxy(11.14,37.37)}{\pgfxy(11.18,30.23)}\pgfcurveto{\pgfxy(11.22,23.09)}{\pgfxy(10.49,15.98)}{\pgfxy(9.00,9.00)}\pgfstroke
\pgfmoveto{\pgfxy(8.77,51.31)}\pgfcurveto{\pgfxy(8.79,46.87)}{\pgfxy(7.87,42.49)}{\pgfxy(6.09,38.43)}\pgfcurveto{\pgfxy(4.74,35.38)}{\pgfxy(2.93,32.57)}{\pgfxy(0.71,30.09)}\pgfstroke
\pgfmoveto{\pgfxy(0.43,30.09)}\pgfcurveto{\pgfxy(2.21,27.86)}{\pgfxy(3.73,25.44)}{\pgfxy(4.95,22.87)}\pgfcurveto{\pgfxy(7.08,18.39)}{\pgfxy(8.29,13.53)}{\pgfxy(8.49,8.58)}\pgfstroke
\pgfmoveto{\pgfxy(20.94,25.42)}\pgfcurveto{\pgfxy(17.74,28.20)}{\pgfxy(13.49,23.54)}{\pgfxy(16.56,20.61)}\pgfcurveto{\pgfxy(19.40,17.89)}{\pgfxy(23.66,21.76)}{\pgfxy(21.22,24.85)}\pgfstroke
\pgfmoveto{\pgfxy(15.71,22.30)}\pgflineto{\pgfxy(16.84,20.75)}\pgfstroke
\pgfmoveto{\pgfxy(16.84,20.75)}\pgflineto{\pgfxy(15.76,23.42)}\pgflineto{\pgfxy(14.63,22.60)}\pgflineto{\pgfxy(16.84,20.75)}\pgfclosepath\pgffill
\pgfmoveto{\pgfxy(16.84,20.75)}\pgflineto{\pgfxy(15.76,23.42)}\pgflineto{\pgfxy(14.63,22.60)}\pgflineto{\pgfxy(16.84,20.75)}\pgfclosepath\pgfstroke
\pgfsetlinewidth{0.30mm}\pgfmoveto{\pgfxy(30.00,9.01)}\pgflineto{\pgfxy(30.00,11.27)}\pgfstroke
\pgfmoveto{\pgfxy(30.00,11.27)}\pgflineto{\pgfxy(29.30,8.47)}\pgflineto{\pgfxy(30.70,8.47)}\pgflineto{\pgfxy(30.00,11.27)}\pgfclosepath\pgffill
\pgfmoveto{\pgfxy(30.00,11.27)}\pgflineto{\pgfxy(29.30,8.47)}\pgflineto{\pgfxy(30.70,8.47)}\pgflineto{\pgfxy(30.00,11.27)}\pgfclosepath\pgfstroke
\pgfmoveto{\pgfxy(30.00,9.01)}\pgflineto{\pgfxy(30.00,11.27)}\pgfstroke
\pgfmoveto{\pgfxy(30.00,11.27)}\pgflineto{\pgfxy(29.30,8.47)}\pgflineto{\pgfxy(30.70,8.47)}\pgflineto{\pgfxy(30.00,11.27)}\pgfclosepath\pgffill
\pgfmoveto{\pgfxy(30.00,11.27)}\pgflineto{\pgfxy(29.30,8.47)}\pgflineto{\pgfxy(30.70,8.47)}\pgflineto{\pgfxy(30.00,11.27)}\pgfclosepath\pgfstroke
\pgfmoveto{\pgfxy(29.86,-0.19)}\pgfcurveto{\pgfxy(29.86,-0.19)}{\pgfxy(29.86,20.18)}{\pgfxy(29.86,20.18)}\pgfstroke
\pgfmoveto{\pgfxy(38.63,24.43)}\pgfcurveto{\pgfxy(41.51,20.24)}{\pgfxy(42.24,14.95)}{\pgfxy(40.62,10.14)}\pgfcurveto{\pgfxy(38.93,5.15)}{\pgfxy(34.92,1.30)}{\pgfxy(29.86,-0.19)}\pgfstroke
\pgfmoveto{\pgfxy(30.00,19.76)}\pgfcurveto{\pgfxy(31.00,16.84)}{\pgfxy(34.45,15.60)}{\pgfxy(37.08,17.21)}\pgfcurveto{\pgfxy(39.48,18.68)}{\pgfxy(40.07,21.91)}{\pgfxy(38.35,24.14)}\pgfstroke
\pgfmoveto{\pgfxy(48.96,28.81)}\pgflineto{\pgfxy(48.82,30.79)}\pgfstroke
\pgfmoveto{\pgfxy(48.82,30.79)}\pgflineto{\pgfxy(48.32,27.95)}\pgflineto{\pgfxy(49.72,28.05)}\pgflineto{\pgfxy(48.82,30.79)}\pgfclosepath\pgffill
\pgfmoveto{\pgfxy(48.82,30.79)}\pgflineto{\pgfxy(48.32,27.95)}\pgflineto{\pgfxy(49.72,28.05)}\pgflineto{\pgfxy(48.82,30.79)}\pgfclosepath\pgfstroke
\pgfmoveto{\pgfxy(51.23,8.44)}\pgfcurveto{\pgfxy(49.66,15.41)}{\pgfxy(48.86,22.52)}{\pgfxy(48.82,29.66)}\pgfcurveto{\pgfxy(48.78,36.79)}{\pgfxy(49.51,43.91)}{\pgfxy(51.00,50.89)}\pgfstroke
\pgfmoveto{\pgfxy(51.23,8.58)}\pgfcurveto{\pgfxy(51.21,13.01)}{\pgfxy(52.13,17.40)}{\pgfxy(53.91,21.46)}\pgfcurveto{\pgfxy(55.26,24.50)}{\pgfxy(57.07,27.32)}{\pgfxy(59.29,29.80)}\pgfstroke
\pgfmoveto{\pgfxy(59.57,29.80)}\pgfcurveto{\pgfxy(57.79,32.02)}{\pgfxy(56.27,34.45)}{\pgfxy(55.05,37.02)}\pgfcurveto{\pgfxy(52.92,41.49)}{\pgfxy(51.71,46.35)}{\pgfxy(51.51,51.31)}\pgfstroke
\pgfmoveto{\pgfxy(39.06,34.47)}\pgfcurveto{\pgfxy(42.26,31.69)}{\pgfxy(46.51,36.35)}{\pgfxy(43.45,39.28)}\pgfcurveto{\pgfxy(40.60,42.00)}{\pgfxy(36.35,38.12)}{\pgfxy(38.78,35.04)}\pgfstroke
\pgfmoveto{\pgfxy(44.29,37.58)}\pgflineto{\pgfxy(43.16,39.14)}\pgfstroke
\pgfmoveto{\pgfxy(43.16,39.14)}\pgflineto{\pgfxy(44.24,36.46)}\pgflineto{\pgfxy(45.38,37.29)}\pgflineto{\pgfxy(43.16,39.14)}\pgfclosepath\pgffill
\pgfmoveto{\pgfxy(43.16,39.14)}\pgflineto{\pgfxy(44.24,36.46)}\pgflineto{\pgfxy(45.38,37.29)}\pgflineto{\pgfxy(43.16,39.14)}\pgfclosepath\pgfstroke
\pgfputat{\pgfxy(30.14,63.19)}{\pgfbox[bottom,left]{\fontsize{7.97}{9.56}\selectfont \makebox[0pt]{$1$}}}
\pgfputat{\pgfxy(54.33,52.44)}{\pgfbox[bottom,left]{\fontsize{7.97}{9.56}\selectfont \makebox[0pt]{$2$}}}
\pgfputat{\pgfxy(62.96,29.10)}{\pgfbox[bottom,left]{\fontsize{7.97}{9.56}\selectfont \makebox[0pt]{$3$}}}
\pgfputat{\pgfxy(53.62,5.19)}{\pgfbox[bottom,left]{\fontsize{7.97}{9.56}\selectfont \makebox[0pt]{$4$}}}
\pgfputat{\pgfxy(30.14,-4.43)}{\pgfbox[bottom,left]{\fontsize{7.97}{9.56}\selectfont \makebox[0pt]{$-1$}}}
\pgfputat{\pgfxy(7.22,4.20)}{\pgfbox[bottom,left]{\fontsize{7.97}{9.56}\selectfont \makebox[0pt]{$-2$}}}
\pgfputat{\pgfxy(-3.53,28.67)}{\pgfbox[bottom,left]{\fontsize{7.97}{9.56}\selectfont \makebox[0pt]{$-3$}}}
\pgfputat{\pgfxy(5.38,52.58)}{\pgfbox[bottom,left]{\fontsize{7.97}{9.56}\selectfont \makebox[0pt]{$-4$}}}
\pgfputat{\pgfxy(29.85,35.32)}{\pgfbox[bottom,left]{\fontsize{7.97}{9.56}\selectfont \makebox[0pt]{$5$}}}
\pgfputat{\pgfxy(23.29,32.81)}{\pgfbox[bottom,left]{\fontsize{7.97}{9.56}\selectfont $6$}}
\pgfputat{\pgfxy(23.08,25.26)}{\pgfbox[bottom,left]{\fontsize{7.97}{9.56}\selectfont $7$}}
\pgfputat{\pgfxy(36.53,32.92)}{\pgfbox[bottom,left]{\fontsize{7.97}{9.56}\selectfont \makebox[0pt][r]{$-7$}}}
\pgfputat{\pgfxy(29.85,22.30)}{\pgfbox[bottom,left]{\fontsize{7.97}{9.56}\selectfont \makebox[0pt]{$-5$}}}
\pgfputat{\pgfxy(36.43,24.95)}{\pgfbox[bottom,left]{\fontsize{7.97}{9.56}\selectfont \makebox[0pt][r]{$-6$}}}
\pgfsetlinewidth{0.60mm}\pgfcircle[stroke]{\pgfxy(110.01,30.00)}{21.16mm}
\pgfcircle[stroke]{\pgfxy(110.01,30.00)}{7.00mm}
\pgfcircle[fill]{\pgfxy(110.01,60.00)}{0.70mm}
\pgfsetlinewidth{0.30mm}\pgfcircle[stroke]{\pgfxy(110.01,60.00)}{0.70mm}
\pgfcircle[fill]{\pgfxy(110.01,0.00)}{0.70mm}
\pgfcircle[stroke]{\pgfxy(110.01,0.00)}{0.70mm}
\pgfcircle[fill]{\pgfxy(140.01,30.00)}{0.70mm}
\pgfcircle[stroke]{\pgfxy(140.01,30.00)}{0.70mm}
\pgfcircle[fill]{\pgfxy(80.01,30.00)}{0.70mm}
\pgfcircle[stroke]{\pgfxy(80.01,30.00)}{0.70mm}
\pgfcircle[fill]{\pgfxy(88.79,51.31)}{0.70mm}
\pgfcircle[stroke]{\pgfxy(88.79,51.31)}{0.70mm}
\pgfcircle[fill]{\pgfxy(131.30,51.14)}{0.70mm}
\pgfcircle[stroke]{\pgfxy(131.30,51.14)}{0.70mm}
\pgfcircle[fill]{\pgfxy(131.30,8.43)}{0.70mm}
\pgfcircle[stroke]{\pgfxy(131.30,8.43)}{0.70mm}
\pgfcircle[fill]{\pgfxy(88.79,8.58)}{0.70mm}
\pgfcircle[stroke]{\pgfxy(88.79,8.58)}{0.70mm}
\pgfcircle[fill]{\pgfxy(110.01,40.00)}{0.70mm}
\pgfcircle[stroke]{\pgfxy(110.01,40.00)}{0.70mm}
\pgfcircle[fill]{\pgfxy(110.01,20.00)}{0.70mm}
\pgfcircle[stroke]{\pgfxy(110.01,20.00)}{0.70mm}
\pgfcircle[fill]{\pgfxy(118.07,35.60)}{0.70mm}
\pgfcircle[stroke]{\pgfxy(118.07,35.60)}{0.70mm}
\pgfcircle[fill]{\pgfxy(118.36,24.85)}{0.70mm}
\pgfcircle[stroke]{\pgfxy(118.36,24.85)}{0.70mm}
\pgfcircle[fill]{\pgfxy(101.38,24.85)}{0.70mm}
\pgfcircle[stroke]{\pgfxy(101.38,24.85)}{0.70mm}
\pgfcircle[fill]{\pgfxy(101.52,35.60)}{0.70mm}
\pgfcircle[stroke]{\pgfxy(101.52,35.60)}{0.70mm}
\pgfputat{\pgfxy(110.15,63.19)}{\pgfbox[bottom,left]{\fontsize{7.97}{9.56}\selectfont \makebox[0pt]{$1$}}}
\pgfputat{\pgfxy(134.34,52.44)}{\pgfbox[bottom,left]{\fontsize{7.97}{9.56}\selectfont \makebox[0pt]{$2$}}}
\pgfputat{\pgfxy(142.97,29.10)}{\pgfbox[bottom,left]{\fontsize{7.97}{9.56}\selectfont \makebox[0pt]{$3$}}}
\pgfputat{\pgfxy(133.63,5.19)}{\pgfbox[bottom,left]{\fontsize{7.97}{9.56}\selectfont \makebox[0pt]{$4$}}}
\pgfputat{\pgfxy(110.15,-4.43)}{\pgfbox[bottom,left]{\fontsize{7.97}{9.56}\selectfont \makebox[0pt]{$-1$}}}
\pgfputat{\pgfxy(87.23,4.20)}{\pgfbox[bottom,left]{\fontsize{7.97}{9.56}\selectfont \makebox[0pt]{$-2$}}}
\pgfputat{\pgfxy(76.48,28.67)}{\pgfbox[bottom,left]{\fontsize{7.97}{9.56}\selectfont \makebox[0pt]{$-3$}}}
\pgfputat{\pgfxy(85.39,52.58)}{\pgfbox[bottom,left]{\fontsize{7.97}{9.56}\selectfont \makebox[0pt]{$-4$}}}
\pgfputat{\pgfxy(109.87,35.32)}{\pgfbox[bottom,left]{\fontsize{7.97}{9.56}\selectfont \makebox[0pt]{$5$}}}
\pgfputat{\pgfxy(103.30,32.81)}{\pgfbox[bottom,left]{\fontsize{7.97}{9.56}\selectfont $6$}}
\pgfputat{\pgfxy(103.10,25.26)}{\pgfbox[bottom,left]{\fontsize{7.97}{9.56}\selectfont $7$}}
\pgfputat{\pgfxy(116.54,32.92)}{\pgfbox[bottom,left]{\fontsize{7.97}{9.56}\selectfont \makebox[0pt][r]{$-7$}}}
\pgfputat{\pgfxy(109.87,22.30)}{\pgfbox[bottom,left]{\fontsize{7.97}{9.56}\selectfont \makebox[0pt]{$-5$}}}
\pgfputat{\pgfxy(116.44,24.95)}{\pgfbox[bottom,left]{\fontsize{7.97}{9.56}\selectfont \makebox[0pt][r]{$-6$}}}
\pgfmoveto{\pgfxy(110.06,59.91)}\pgfcurveto{\pgfxy(110.06,59.91)}{\pgfxy(110.06,40.01)}{\pgfxy(110.06,40.01)}\pgfstroke
\pgfmoveto{\pgfxy(110.06,40.15)}\pgfcurveto{\pgfxy(110.47,43.00)}{\pgfxy(113.28,44.86)}{\pgfxy(116.06,44.10)}\pgfcurveto{\pgfxy(119.81,43.09)}{\pgfxy(121.10,38.42)}{\pgfxy(118.40,35.62)}\pgfstroke
\pgfmoveto{\pgfxy(118.40,35.62)}\pgfcurveto{\pgfxy(118.40,35.62)}{\pgfxy(131.28,51.13)}{\pgfxy(131.28,51.13)}\pgfstroke
\pgfmoveto{\pgfxy(131.28,50.98)}\pgfcurveto{\pgfxy(131.76,40.78)}{\pgfxy(129.85,30.61)}{\pgfxy(125.71,21.28)}\pgfcurveto{\pgfxy(122.13,13.20)}{\pgfxy(116.96,5.94)}{\pgfxy(110.50,-0.08)}\pgfstroke
\pgfmoveto{\pgfxy(139.91,30.06)}\pgfcurveto{\pgfxy(134.95,29.81)}{\pgfxy(130.57,26.74)}{\pgfxy(128.64,22.16)}\pgfcurveto{\pgfxy(126.77,17.69)}{\pgfxy(127.56,12.54)}{\pgfxy(130.69,8.84)}\pgfstroke
\pgfmoveto{\pgfxy(133.84,21.37)}\pgflineto{\pgfxy(132.65,18.42)}\pgfstroke
\pgfmoveto{\pgfxy(132.65,18.42)}\pgflineto{\pgfxy(134.35,20.76)}\pgflineto{\pgfxy(133.05,21.28)}\pgflineto{\pgfxy(132.65,18.42)}\pgfclosepath\pgffill
\pgfmoveto{\pgfxy(132.65,18.42)}\pgflineto{\pgfxy(134.35,20.76)}\pgflineto{\pgfxy(133.05,21.28)}\pgflineto{\pgfxy(132.65,18.42)}\pgfclosepath\pgfstroke
\pgfmoveto{\pgfxy(131.28,8.69)}\pgfcurveto{\pgfxy(131.10,12.76)}{\pgfxy(131.85,16.81)}{\pgfxy(133.47,20.55)}\pgfcurveto{\pgfxy(134.99,24.05)}{\pgfxy(137.23,27.19)}{\pgfxy(140.05,29.76)}\pgfstroke
\pgfmoveto{\pgfxy(123.34,23.82)}\pgflineto{\pgfxy(122.36,21.44)}\pgfstroke
\pgfmoveto{\pgfxy(122.36,21.44)}\pgflineto{\pgfxy(124.07,23.76)}\pgflineto{\pgfxy(122.78,24.29)}\pgflineto{\pgfxy(122.36,21.44)}\pgfclosepath\pgffill
\pgfmoveto{\pgfxy(122.36,21.44)}\pgflineto{\pgfxy(124.07,23.76)}\pgflineto{\pgfxy(122.78,24.29)}\pgflineto{\pgfxy(122.36,21.44)}\pgfclosepath\pgfstroke
\pgfmoveto{\pgfxy(118.55,24.79)}\pgfcurveto{\pgfxy(116.82,22.10)}{\pgfxy(120.66,19.36)}{\pgfxy(122.64,21.86)}\pgfcurveto{\pgfxy(124.81,24.60)}{\pgfxy(120.69,27.83)}{\pgfxy(118.55,25.08)}\pgfstroke
\pgfmoveto{\pgfxy(109.91,50.83)}\pgflineto{\pgfxy(110.06,49.23)}\pgfstroke
\pgfmoveto{\pgfxy(109.91,51.71)}\pgflineto{\pgfxy(110.06,48.20)}\pgfstroke
\pgfmoveto{\pgfxy(110.06,48.20)}\pgflineto{\pgfxy(110.64,51.03)}\pgflineto{\pgfxy(109.24,50.97)}\pgflineto{\pgfxy(110.06,48.20)}\pgfclosepath\pgffill
\pgfmoveto{\pgfxy(110.06,48.20)}\pgflineto{\pgfxy(110.64,51.03)}\pgflineto{\pgfxy(109.24,50.97)}\pgflineto{\pgfxy(110.06,48.20)}\pgfclosepath\pgfstroke
\pgfmoveto{\pgfxy(110.15,0.02)}\pgfcurveto{\pgfxy(110.15,0.02)}{\pgfxy(110.15,19.92)}{\pgfxy(110.15,19.92)}\pgfstroke
\pgfmoveto{\pgfxy(110.15,19.77)}\pgfcurveto{\pgfxy(109.75,16.92)}{\pgfxy(106.93,15.07)}{\pgfxy(104.15,15.82)}\pgfcurveto{\pgfxy(100.40,16.84)}{\pgfxy(99.11,21.51)}{\pgfxy(101.81,24.31)}\pgfstroke
\pgfmoveto{\pgfxy(101.81,24.31)}\pgfcurveto{\pgfxy(101.81,24.31)}{\pgfxy(88.94,8.80)}{\pgfxy(88.94,8.80)}\pgfstroke
\pgfmoveto{\pgfxy(88.94,8.94)}\pgfcurveto{\pgfxy(88.45,19.14)}{\pgfxy(90.36,29.31)}{\pgfxy(94.50,38.65)}\pgfcurveto{\pgfxy(98.08,46.72)}{\pgfxy(103.25,53.99)}{\pgfxy(109.71,60.01)}\pgfstroke
\pgfmoveto{\pgfxy(80.30,29.87)}\pgfcurveto{\pgfxy(85.27,30.12)}{\pgfxy(89.65,33.19)}{\pgfxy(91.57,37.77)}\pgfcurveto{\pgfxy(93.45,42.24)}{\pgfxy(92.66,47.39)}{\pgfxy(89.52,51.08)}\pgfstroke
\pgfmoveto{\pgfxy(86.37,38.56)}\pgflineto{\pgfxy(87.56,41.50)}\pgfstroke
\pgfmoveto{\pgfxy(87.56,41.50)}\pgflineto{\pgfxy(85.87,39.17)}\pgflineto{\pgfxy(87.16,38.65)}\pgflineto{\pgfxy(87.56,41.50)}\pgfclosepath\pgffill
\pgfmoveto{\pgfxy(87.56,41.50)}\pgflineto{\pgfxy(85.87,39.17)}\pgflineto{\pgfxy(87.16,38.65)}\pgflineto{\pgfxy(87.56,41.50)}\pgfclosepath\pgfstroke
\pgfmoveto{\pgfxy(88.94,51.23)}\pgfcurveto{\pgfxy(89.11,47.17)}{\pgfxy(88.36,43.11)}{\pgfxy(86.74,39.38)}\pgfcurveto{\pgfxy(85.23,35.87)}{\pgfxy(82.98,32.73)}{\pgfxy(80.16,30.16)}\pgfstroke
\pgfmoveto{\pgfxy(96.87,36.11)}\pgflineto{\pgfxy(97.86,38.49)}\pgfstroke
\pgfmoveto{\pgfxy(97.86,38.49)}\pgflineto{\pgfxy(96.14,36.17)}\pgflineto{\pgfxy(97.44,35.63)}\pgflineto{\pgfxy(97.86,38.49)}\pgfclosepath\pgffill
\pgfmoveto{\pgfxy(97.86,38.49)}\pgflineto{\pgfxy(96.14,36.17)}\pgflineto{\pgfxy(97.44,35.63)}\pgflineto{\pgfxy(97.86,38.49)}\pgfclosepath\pgfstroke
\pgfmoveto{\pgfxy(101.67,35.14)}\pgfcurveto{\pgfxy(103.40,37.83)}{\pgfxy(99.56,40.57)}{\pgfxy(97.57,38.06)}\pgfcurveto{\pgfxy(95.41,35.33)}{\pgfxy(99.52,32.09)}{\pgfxy(101.67,34.84)}\pgfstroke
\pgfputat{\pgfxy(0.00,61.00)}{\pgfbox[bottom,left]{\fontsize{9.96}{11.95}\selectfont \makebox[0pt]{$\pi$}}}
\pgfputat{\pgfxy(80.00,61.00)}{\pgfbox[bottom,left]{\fontsize{9.96}{11.95}\selectfont \makebox[0pt]{$\sigma$}}}
\end{pgfpicture}%
\else
\setlength{\unitlength}{0.70 mm}%
\begin{picture}(154.89,77.34)(-8.06,-8.15)
\thicklines\rqbezier(30.00,60.23)(-0.23,60.23)(-0.23,30.00)(0.70711)\rqbezier(-0.23,30.00)(-0.23,-0.23)(30.00,-0.23)(0.70711)\rqbezier(30.00,-0.23)(60.23,-0.23)(60.23,30.00)(0.70711)\rqbezier(60.23,30.00)(60.23,60.23)(30.00,60.23)(0.70711)
\thicklines\rqbezier(30.00,40.00)(20.00,40.00)(20.00,30.00)(0.70711)\rqbezier(20.00,30.00)(20.00,20.00)(30.00,20.00)(0.70711)\rqbezier(30.00,20.00)(40.00,20.00)(40.00,30.00)(0.70711)\rqbezier(40.00,30.00)(40.00,40.00)(30.00,40.00)(0.70711)
\thinlines\rqbezier(30.00,61.00)(29.00,61.00)(29.00,60.00)(0.70711)\rqbezier(29.00,60.00)(29.00,59.00)(30.00,59.00)(0.70711)\rqbezier(30.00,59.00)(31.00,59.00)(31.00,60.00)(0.70711)\rqbezier(31.00,60.00)(31.00,61.00)(30.00,61.00)(0.70711)
\thinlines\rqbezier(30.00,1.00)(29.00,1.00)(29.00,0.00)(0.70711)\rqbezier(29.00,0.00)(29.00,-1.00)(30.00,-1.00)(0.70711)\rqbezier(30.00,-1.00)(31.00,-1.00)(31.00,0.00)(0.70711)\rqbezier(31.00,0.00)(31.00,1.00)(30.00,1.00)(0.70711)
\thinlines\rqbezier(60.00,31.00)(59.00,31.00)(59.00,30.00)(0.70711)\rqbezier(59.00,30.00)(59.00,29.00)(60.00,29.00)(0.70711)\rqbezier(60.00,29.00)(61.00,29.00)(61.00,30.00)(0.70711)\rqbezier(61.00,30.00)(61.00,31.00)(60.00,31.00)(0.70711)
\thinlines\rqbezier(0.00,31.00)(-1.00,31.00)(-1.00,30.00)(0.70711)\rqbezier(-1.00,30.00)(-1.00,29.00)(0.00,29.00)(0.70711)\rqbezier(0.00,29.00)(1.00,29.00)(1.00,30.00)(0.70711)\rqbezier(1.00,30.00)(1.00,31.00)(0.00,31.00)(0.70711)
\thinlines\rqbezier(8.77,52.31)(7.77,52.31)(7.77,51.31)(0.70711)\rqbezier(7.77,51.31)(7.77,50.31)(8.77,50.31)(0.70711)\rqbezier(8.77,50.31)(9.77,50.31)(9.77,51.31)(0.70711)\rqbezier(9.77,51.31)(9.77,52.31)(8.77,52.31)(0.70711)
\thinlines\rqbezier(51.28,52.14)(50.28,52.14)(50.28,51.14)(0.70711)\rqbezier(50.28,51.14)(50.28,50.14)(51.28,50.14)(0.70711)\rqbezier(51.28,50.14)(52.28,50.14)(52.28,51.14)(0.70711)\rqbezier(52.28,51.14)(52.28,52.14)(51.28,52.14)(0.70711)
\thinlines\rqbezier(51.28,9.43)(50.28,9.43)(50.28,8.43)(0.70711)\rqbezier(50.28,8.43)(50.28,7.43)(51.28,7.43)(0.70711)\rqbezier(51.28,7.43)(52.28,7.43)(52.28,8.43)(0.70711)\rqbezier(52.28,8.43)(52.28,9.43)(51.28,9.43)(0.70711)
\thinlines\rqbezier(8.77,9.58)(7.77,9.58)(7.77,8.58)(0.70711)\rqbezier(7.77,8.58)(7.77,7.58)(8.77,7.58)(0.70711)\rqbezier(8.77,7.58)(9.77,7.58)(9.77,8.58)(0.70711)\rqbezier(9.77,8.58)(9.77,9.58)(8.77,9.58)(0.70711)
\thinlines\rqbezier(30.00,41.00)(29.00,41.00)(29.00,40.00)(0.70711)\rqbezier(29.00,40.00)(29.00,39.00)(30.00,39.00)(0.70711)\rqbezier(30.00,39.00)(31.00,39.00)(31.00,40.00)(0.70711)\rqbezier(31.00,40.00)(31.00,41.00)(30.00,41.00)(0.70711)
\thinlines\rqbezier(30.00,21.00)(29.00,21.00)(29.00,20.00)(0.70711)\rqbezier(29.00,20.00)(29.00,19.00)(30.00,19.00)(0.70711)\rqbezier(30.00,19.00)(31.00,19.00)(31.00,20.00)(0.70711)\rqbezier(31.00,20.00)(31.00,21.00)(30.00,21.00)(0.70711)
\thinlines\rqbezier(38.06,36.60)(37.06,36.60)(37.06,35.60)(0.70711)\rqbezier(37.06,35.60)(37.06,34.60)(38.06,34.60)(0.70711)\rqbezier(38.06,34.60)(39.06,34.60)(39.06,35.60)(0.70711)\rqbezier(39.06,35.60)(39.06,36.60)(38.06,36.60)(0.70711)
\thinlines\rqbezier(38.34,25.85)(37.34,25.85)(37.34,24.85)(0.70711)\rqbezier(37.34,24.85)(37.34,23.85)(38.34,23.85)(0.70711)\rqbezier(38.34,23.85)(39.34,23.85)(39.34,24.85)(0.70711)\rqbezier(39.34,24.85)(39.34,25.85)(38.34,25.85)(0.70711)
\thinlines\rqbezier(21.37,25.85)(20.37,25.85)(20.37,24.85)(0.70711)\rqbezier(20.37,24.85)(20.37,23.85)(21.37,23.85)(0.70711)\rqbezier(21.37,23.85)(22.37,23.85)(22.37,24.85)(0.70711)\rqbezier(22.37,24.85)(22.37,25.85)(21.37,25.85)(0.70711)
\thinlines\rqbezier(21.51,36.60)(20.51,36.60)(20.51,35.60)(0.70711)\rqbezier(20.51,35.60)(20.51,34.60)(21.51,34.60)(0.70711)\rqbezier(21.51,34.60)(22.51,34.60)(22.51,35.60)(0.70711)\rqbezier(22.51,35.60)(22.51,36.60)(21.51,36.60)(0.70711)
\thinlines\cbezier(20.94,25.42)(17.74,28.20)(13.49,23.54)(16.56,20.61)\cbezier(16.56,20.61)(19.40,17.89)(23.66,21.76)(21.22,24.85)
\thinlines\cbezier(8.77,51.31)(8.79,46.87)(7.87,42.49)(6.09,38.43)\cbezier(6.09,38.43)(4.74,35.38)(2.93,32.57)(0.71,30.09)
\thinlines\cbezier(0.43,30.09)(2.21,27.86)(3.73,25.44)(4.95,22.87)\cbezier(4.95,22.87)(7.08,18.39)(8.29,13.53)(8.49,8.58)
\thinlines\cbezier(8.77,51.45)(10.34,44.48)(11.14,37.37)(11.18,30.23)\cbezier(11.18,30.23)(11.22,23.09)(10.49,15.98)(9.00,9.00)
\thinlines\cbezier(30.00,40.13)(29.00,43.05)(25.55,44.29)(22.92,42.68)\cbezier(22.92,42.68)(20.52,41.21)(19.93,37.97)(21.65,35.74)
\thinlines\cbezier(21.37,35.46)(18.49,39.65)(17.76,44.94)(19.38,49.75)\cbezier(19.38,49.75)(21.07,54.74)(25.08,58.59)(30.14,60.08)
\thinlines\cbezier(30.14,60.08)(30.14,60.08)(30.14,39.71)(30.14,39.71)
\thinlines\lbezier(11.04,31.08)(11.18,29.10)
\thinlines\lbezier(11.18,29.10)(11.68,31.94)\lbezier(11.68,31.94)(10.28,31.84)\lbezier(10.28,31.84)(11.18,29.10)
\thinlines\lbezier(15.71,22.30)(16.84,20.75)
\thinlines\lbezier(16.84,20.75)(15.76,23.42)\lbezier(15.76,23.42)(14.63,22.60)\lbezier(14.63,22.60)(16.84,20.75)
\thinlines\lbezier(30.00,50.88)(30.00,48.62)
\thinlines\lbezier(30.00,48.62)(30.70,51.42)\lbezier(30.70,51.42)(29.30,51.42)\lbezier(29.30,51.42)(30.00,48.62)
\thinlines\lbezier(30.00,50.88)(30.00,48.62)
\thinlines\lbezier(30.00,48.62)(30.70,51.42)\lbezier(30.70,51.42)(29.30,51.42)\lbezier(29.30,51.42)(30.00,48.62)
\thinlines\cbezier(30.14,60.08)(30.14,60.08)(30.14,39.71)(30.14,39.71)
\thinlines\cbezier(21.37,35.46)(18.49,39.65)(17.76,44.94)(19.38,49.75)\cbezier(19.38,49.75)(21.07,54.74)(25.08,58.59)(30.14,60.08)
\thinlines\cbezier(30.00,40.13)(29.00,43.05)(25.55,44.29)(22.92,42.68)\cbezier(22.92,42.68)(20.52,41.21)(19.93,37.97)(21.65,35.74)
\thinlines\lbezier(11.04,31.08)(11.18,29.10)
\thinlines\lbezier(11.18,29.10)(11.68,31.94)\lbezier(11.68,31.94)(10.28,31.84)\lbezier(10.28,31.84)(11.18,29.10)
\thinlines\cbezier(8.77,51.45)(10.34,44.48)(11.14,37.37)(11.18,30.23)\cbezier(11.18,30.23)(11.22,23.09)(10.49,15.98)(9.00,9.00)
\thinlines\cbezier(8.77,51.31)(8.79,46.87)(7.87,42.49)(6.09,38.43)\cbezier(6.09,38.43)(4.74,35.38)(2.93,32.57)(0.71,30.09)
\thinlines\cbezier(0.43,30.09)(2.21,27.86)(3.73,25.44)(4.95,22.87)\cbezier(4.95,22.87)(7.08,18.39)(8.29,13.53)(8.49,8.58)
\thinlines\cbezier(20.94,25.42)(17.74,28.20)(13.49,23.54)(16.56,20.61)\cbezier(16.56,20.61)(19.40,17.89)(23.66,21.76)(21.22,24.85)
\thinlines\lbezier(15.71,22.30)(16.84,20.75)
\thinlines\lbezier(16.84,20.75)(15.76,23.42)\lbezier(15.76,23.42)(14.63,22.60)\lbezier(14.63,22.60)(16.84,20.75)
\thinlines\lbezier(30.00,9.01)(30.00,11.27)
\thinlines\lbezier(30.00,11.27)(29.30,8.47)\lbezier(29.30,8.47)(30.70,8.47)\lbezier(30.70,8.47)(30.00,11.27)
\thinlines\lbezier(30.00,9.01)(30.00,11.27)
\thinlines\lbezier(30.00,11.27)(29.30,8.47)\lbezier(29.30,8.47)(30.70,8.47)\lbezier(30.70,8.47)(30.00,11.27)
\thinlines\cbezier(29.86,-0.19)(29.86,-0.19)(29.86,20.18)(29.86,20.18)
\thinlines\cbezier(38.63,24.43)(41.51,20.24)(42.24,14.95)(40.62,10.14)\cbezier(40.62,10.14)(38.93,5.15)(34.92,1.30)(29.86,-0.19)
\thinlines\cbezier(30.00,19.76)(31.00,16.84)(34.45,15.60)(37.08,17.21)\cbezier(37.08,17.21)(39.48,18.68)(40.07,21.91)(38.35,24.14)
\thinlines\lbezier(48.96,28.81)(48.82,30.79)
\thinlines\lbezier(48.82,30.79)(48.32,27.95)\lbezier(48.32,27.95)(49.72,28.05)\lbezier(49.72,28.05)(48.82,30.79)
\thinlines\cbezier(51.23,8.44)(49.66,15.41)(48.86,22.52)(48.82,29.66)\cbezier(48.82,29.66)(48.78,36.79)(49.51,43.91)(51.00,50.89)
\thinlines\cbezier(51.23,8.58)(51.21,13.01)(52.13,17.40)(53.91,21.46)\cbezier(53.91,21.46)(55.26,24.50)(57.07,27.32)(59.29,29.80)
\thinlines\cbezier(59.57,29.80)(57.79,32.02)(56.27,34.45)(55.05,37.02)\cbezier(55.05,37.02)(52.92,41.49)(51.71,46.35)(51.51,51.31)
\thinlines\cbezier(39.06,34.47)(42.26,31.69)(46.51,36.35)(43.45,39.28)\cbezier(43.45,39.28)(40.60,42.00)(36.35,38.12)(38.78,35.04)
\thinlines\lbezier(44.29,37.58)(43.16,39.14)
\thinlines\lbezier(43.16,39.14)(44.24,36.46)\lbezier(44.24,36.46)(45.38,37.29)\lbezier(45.38,37.29)(43.16,39.14)
\put(30.14,63.19){\fontsize{7.97}{9.56}\selectfont \makebox[0pt]{$1$}}
\put(54.33,52.44){\fontsize{7.97}{9.56}\selectfont \makebox[0pt]{$2$}}
\put(62.96,29.10){\fontsize{7.97}{9.56}\selectfont \makebox[0pt]{$3$}}
\put(53.62,5.19){\fontsize{7.97}{9.56}\selectfont \makebox[0pt]{$4$}}
\put(30.14,-4.43){\fontsize{7.97}{9.56}\selectfont \makebox[0pt]{$-1$}}
\put(7.22,4.20){\fontsize{7.97}{9.56}\selectfont \makebox[0pt]{$-2$}}
\put(-3.53,28.67){\fontsize{7.97}{9.56}\selectfont \makebox[0pt]{$-3$}}
\put(5.38,52.58){\fontsize{7.97}{9.56}\selectfont \makebox[0pt]{$-4$}}
\put(29.85,35.32){\fontsize{7.97}{9.56}\selectfont \makebox[0pt]{$5$}}
\put(23.29,32.81){\fontsize{7.97}{9.56}\selectfont $6$}
\put(23.08,25.26){\fontsize{7.97}{9.56}\selectfont $7$}
\put(36.53,32.92){\fontsize{7.97}{9.56}\selectfont \makebox[0pt][r]{$-7$}}
\put(29.85,22.30){\fontsize{7.97}{9.56}\selectfont \makebox[0pt]{$-5$}}
\put(36.43,24.95){\fontsize{7.97}{9.56}\selectfont \makebox[0pt][r]{$-6$}}
\thicklines\rqbezier(110.01,60.23)(79.78,60.23)(79.78,30.00)(0.70711)\rqbezier(79.78,30.00)(79.78,-0.23)(110.01,-0.23)(0.70711)\rqbezier(110.01,-0.23)(140.25,-0.23)(140.25,30.00)(0.70711)\rqbezier(140.25,30.00)(140.25,60.23)(110.01,60.23)(0.70711)
\thicklines\rqbezier(110.01,40.00)(100.01,40.00)(100.01,30.00)(0.70711)\rqbezier(100.01,30.00)(100.01,20.00)(110.01,20.00)(0.70711)\rqbezier(110.01,20.00)(120.01,20.00)(120.01,30.00)(0.70711)\rqbezier(120.01,30.00)(120.01,40.00)(110.01,40.00)(0.70711)
\thinlines\rqbezier(110.01,61.00)(109.01,61.00)(109.01,60.00)(0.70711)\rqbezier(109.01,60.00)(109.01,59.00)(110.01,59.00)(0.70711)\rqbezier(110.01,59.00)(111.01,59.00)(111.01,60.00)(0.70711)\rqbezier(111.01,60.00)(111.01,61.00)(110.01,61.00)(0.70711)
\thinlines\rqbezier(110.01,1.00)(109.01,1.00)(109.01,0.00)(0.70711)\rqbezier(109.01,0.00)(109.01,-1.00)(110.01,-1.00)(0.70711)\rqbezier(110.01,-1.00)(111.01,-1.00)(111.01,0.00)(0.70711)\rqbezier(111.01,0.00)(111.01,1.00)(110.01,1.00)(0.70711)
\thinlines\rqbezier(140.01,31.00)(139.01,31.00)(139.01,30.00)(0.70711)\rqbezier(139.01,30.00)(139.01,29.00)(140.01,29.00)(0.70711)\rqbezier(140.01,29.00)(141.01,29.00)(141.01,30.00)(0.70711)\rqbezier(141.01,30.00)(141.01,31.00)(140.01,31.00)(0.70711)
\thinlines\rqbezier(80.01,31.00)(79.01,31.00)(79.01,30.00)(0.70711)\rqbezier(79.01,30.00)(79.01,29.00)(80.01,29.00)(0.70711)\rqbezier(80.01,29.00)(81.01,29.00)(81.01,30.00)(0.70711)\rqbezier(81.01,30.00)(81.01,31.00)(80.01,31.00)(0.70711)
\thinlines\rqbezier(88.79,52.31)(87.79,52.31)(87.79,51.31)(0.70711)\rqbezier(87.79,51.31)(87.79,50.31)(88.79,50.31)(0.70711)\rqbezier(88.79,50.31)(89.79,50.31)(89.79,51.31)(0.70711)\rqbezier(89.79,51.31)(89.79,52.31)(88.79,52.31)(0.70711)
\thinlines\rqbezier(131.30,52.14)(130.30,52.14)(130.30,51.14)(0.70711)\rqbezier(130.30,51.14)(130.30,50.14)(131.30,50.14)(0.70711)\rqbezier(131.30,50.14)(132.30,50.14)(132.30,51.14)(0.70711)\rqbezier(132.30,51.14)(132.30,52.14)(131.30,52.14)(0.70711)
\thinlines\rqbezier(131.30,9.43)(130.30,9.43)(130.30,8.43)(0.70711)\rqbezier(130.30,8.43)(130.30,7.43)(131.30,7.43)(0.70711)\rqbezier(131.30,7.43)(132.30,7.43)(132.30,8.43)(0.70711)\rqbezier(132.30,8.43)(132.30,9.43)(131.30,9.43)(0.70711)
\thinlines\rqbezier(88.79,9.58)(87.79,9.58)(87.79,8.58)(0.70711)\rqbezier(87.79,8.58)(87.79,7.58)(88.79,7.58)(0.70711)\rqbezier(88.79,7.58)(89.79,7.58)(89.79,8.58)(0.70711)\rqbezier(89.79,8.58)(89.79,9.58)(88.79,9.58)(0.70711)
\thinlines\rqbezier(110.01,41.00)(109.01,41.00)(109.01,40.00)(0.70711)\rqbezier(109.01,40.00)(109.01,39.00)(110.01,39.00)(0.70711)\rqbezier(110.01,39.00)(111.01,39.00)(111.01,40.00)(0.70711)\rqbezier(111.01,40.00)(111.01,41.00)(110.01,41.00)(0.70711)
\thinlines\rqbezier(110.01,21.00)(109.01,21.00)(109.01,20.00)(0.70711)\rqbezier(109.01,20.00)(109.01,19.00)(110.01,19.00)(0.70711)\rqbezier(110.01,19.00)(111.01,19.00)(111.01,20.00)(0.70711)\rqbezier(111.01,20.00)(111.01,21.00)(110.01,21.00)(0.70711)
\thinlines\rqbezier(118.07,36.60)(117.07,36.60)(117.07,35.60)(0.70711)\rqbezier(117.07,35.60)(117.07,34.60)(118.07,34.60)(0.70711)\rqbezier(118.07,34.60)(119.07,34.60)(119.07,35.60)(0.70711)\rqbezier(119.07,35.60)(119.07,36.60)(118.07,36.60)(0.70711)
\thinlines\rqbezier(118.36,25.85)(117.36,25.85)(117.36,24.85)(0.70711)\rqbezier(117.36,24.85)(117.36,23.85)(118.36,23.85)(0.70711)\rqbezier(118.36,23.85)(119.36,23.85)(119.36,24.85)(0.70711)\rqbezier(119.36,24.85)(119.36,25.85)(118.36,25.85)(0.70711)
\thinlines\rqbezier(101.38,25.85)(100.38,25.85)(100.38,24.85)(0.70711)\rqbezier(100.38,24.85)(100.38,23.85)(101.38,23.85)(0.70711)\rqbezier(101.38,23.85)(102.38,23.85)(102.38,24.85)(0.70711)\rqbezier(102.38,24.85)(102.38,25.85)(101.38,25.85)(0.70711)
\thinlines\rqbezier(101.52,36.60)(100.52,36.60)(100.52,35.60)(0.70711)\rqbezier(100.52,35.60)(100.52,34.60)(101.52,34.60)(0.70711)\rqbezier(101.52,34.60)(102.52,34.60)(102.52,35.60)(0.70711)\rqbezier(102.52,35.60)(102.52,36.60)(101.52,36.60)(0.70711)
\put(110.15,63.19){\fontsize{7.97}{9.56}\selectfont \makebox[0pt]{$1$}}
\put(134.34,52.44){\fontsize{7.97}{9.56}\selectfont \makebox[0pt]{$2$}}
\put(142.97,29.10){\fontsize{7.97}{9.56}\selectfont \makebox[0pt]{$3$}}
\put(133.63,5.19){\fontsize{7.97}{9.56}\selectfont \makebox[0pt]{$4$}}
\put(110.15,-4.43){\fontsize{7.97}{9.56}\selectfont \makebox[0pt]{$-1$}}
\put(87.23,4.20){\fontsize{7.97}{9.56}\selectfont \makebox[0pt]{$-2$}}
\put(76.48,28.67){\fontsize{7.97}{9.56}\selectfont \makebox[0pt]{$-3$}}
\put(85.39,52.58){\fontsize{7.97}{9.56}\selectfont \makebox[0pt]{$-4$}}
\put(109.87,35.32){\fontsize{7.97}{9.56}\selectfont \makebox[0pt]{$5$}}
\put(103.30,32.81){\fontsize{7.97}{9.56}\selectfont $6$}
\put(103.10,25.26){\fontsize{7.97}{9.56}\selectfont $7$}
\put(116.54,32.92){\fontsize{7.97}{9.56}\selectfont \makebox[0pt][r]{$-7$}}
\put(109.87,22.30){\fontsize{7.97}{9.56}\selectfont \makebox[0pt]{$-5$}}
\put(116.44,24.95){\fontsize{7.97}{9.56}\selectfont \makebox[0pt][r]{$-6$}}
\thinlines\cbezier(110.06,59.91)(110.06,59.91)(110.06,40.01)(110.06,40.01)
\thinlines\cbezier(110.06,40.15)(110.47,43.00)(113.28,44.86)(116.06,44.10)\cbezier(116.06,44.10)(119.81,43.09)(121.10,38.42)(118.40,35.62)
\thinlines\cbezier(118.40,35.62)(118.40,35.62)(131.28,51.13)(131.28,51.13)
\thinlines\cbezier(131.28,50.98)(131.76,40.78)(129.85,30.61)(125.71,21.28)\cbezier(125.71,21.28)(122.13,13.20)(116.96,5.94)(110.50,-0.08)
\thinlines\cbezier(139.91,30.06)(134.95,29.81)(130.57,26.74)(128.64,22.16)\cbezier(128.64,22.16)(126.77,17.69)(127.56,12.54)(130.69,8.84)
\thinlines\lbezier(133.84,21.37)(132.65,18.42)
\thinlines\lbezier(132.65,18.42)(134.35,20.76)\lbezier(134.35,20.76)(133.05,21.28)\lbezier(133.05,21.28)(132.65,18.42)
\thinlines\cbezier(131.28,8.69)(131.10,12.76)(131.85,16.81)(133.47,20.55)\cbezier(133.47,20.55)(134.99,24.05)(137.23,27.19)(140.05,29.76)
\thinlines\lbezier(123.34,23.82)(122.36,21.44)
\thinlines\lbezier(122.36,21.44)(124.07,23.76)\lbezier(124.07,23.76)(122.78,24.29)\lbezier(122.78,24.29)(122.36,21.44)
\thinlines\cbezier(118.55,24.79)(116.82,22.10)(120.66,19.36)(122.64,21.86)\cbezier(122.64,21.86)(124.81,24.60)(120.69,27.83)(118.55,25.08)
\thinlines\lbezier(109.91,50.83)(110.06,49.23)
\thinlines\lbezier(109.91,51.71)(110.06,48.20)
\thinlines\lbezier(110.06,48.20)(110.64,51.03)\lbezier(110.64,51.03)(109.24,50.97)\lbezier(109.24,50.97)(110.06,48.20)
\thinlines\cbezier(110.15,0.02)(110.15,0.02)(110.15,19.92)(110.15,19.92)
\thinlines\cbezier(110.15,19.77)(109.75,16.92)(106.93,15.07)(104.15,15.82)\cbezier(104.15,15.82)(100.40,16.84)(99.11,21.51)(101.81,24.31)
\thinlines\cbezier(101.81,24.31)(101.81,24.31)(88.94,8.80)(88.94,8.80)
\thinlines\cbezier(88.94,8.94)(88.45,19.14)(90.36,29.31)(94.50,38.65)\cbezier(94.50,38.65)(98.08,46.72)(103.25,53.99)(109.71,60.01)
\thinlines\cbezier(80.30,29.87)(85.27,30.12)(89.65,33.19)(91.57,37.77)\cbezier(91.57,37.77)(93.45,42.24)(92.66,47.39)(89.52,51.08)
\thinlines\lbezier(86.37,38.56)(87.56,41.50)
\thinlines\lbezier(87.56,41.50)(85.87,39.17)\lbezier(85.87,39.17)(87.16,38.65)\lbezier(87.16,38.65)(87.56,41.50)
\thinlines\cbezier(88.94,51.23)(89.11,47.17)(88.36,43.11)(86.74,39.38)\cbezier(86.74,39.38)(85.23,35.87)(82.98,32.73)(80.16,30.16)
\thinlines\lbezier(96.87,36.11)(97.86,38.49)
\thinlines\lbezier(97.86,38.49)(96.14,36.17)\lbezier(96.14,36.17)(97.44,35.63)\lbezier(97.44,35.63)(97.86,38.49)
\thinlines\cbezier(101.67,35.14)(103.40,37.83)(99.56,40.57)(97.57,38.06)\cbezier(97.57,38.06)(95.41,35.33)(99.52,32.09)(101.67,34.84)
\put(0.00,61.00){\fontsize{9.96}{11.95}\selectfont \makebox[0pt]{$\pi$}}
\put(80.00,61.00){\fontsize{9.96}{11.95}\selectfont \makebox[0pt]{$\sigma$}}
\end{picture}%
\fi
\caption{$\pi=((1~5~6))((2~3~4))$ and $\sigma=[1~5~{-7}~2]((3~4))$ in $\sb(4,3)$}
\label{fig:annulus}
\end{figure}

A paired nonzero cycle $((a_1~a_2 \dots a_k))$ is called \emph{connected} if the
set $\{a_1,\dots,a_k\}$ intersects with both $\set{\pm 1, \dots, \pm p}$ and
$\set{\pm (p+1), \dots, \pm (p+q)}$, and \emph{disconnected} otherwise. A zero
cycle is always considered to be disconnected.  For $\sigma \in \sb(p,q)$, the
\emph{connectivity} of $\sigma$ is the number of connected paired nonzero cycles
of $\sigma$.

We say that a maximal chain $C = \{\epsilon =\pi_0 < \pi_1< \cdots < \pi_{p+q}
=\gamma_{p,q}\}$ of $\sb(p,q)$ is {\em disconnected} if the connectivity of each
$\pi_i$ is zero. Otherwise, $C$ is called {\em connected}.  Denote by
$\CM(\sb(p,q))$ the set of connected maximal chains of $\sb(p,q)$.

For a maximal chain $C = \{\pi_0 < \pi_1< \cdots < \pi_{n} \}$ of the interval
$[\pi_0,\pi_n]$, we define $\varphi(C) = (\tau_{1},\tau_{2},\dots,\tau_{n})$,
where $\tau_i=\pi_{i}^{-1}\pi_{i+1}$. Note that each $\tau_i$ is a type $B$
transposition and $\pi_i=\tau_1\tau_2\cdots\tau_i$ for all $i=1,2,\dots,n$.

\begin{lem}\label{lem:linked}
  If $C$ is a connected maximal chain of $\sb(p,q)$, then $\varphi(C)$ has no
  transpositions of the form $\epsilon_i=[i]$ and has at least one connected
  transposition.  If $C$ is a disconnected maximal chain of $\sb(p,q)$, then
  $\varphi(C)$ has only disconnected transpositions.
\end{lem}
\begin{proof}
  By \eqref{eq:rankB}, $\sigma$ covers $\pi$ in $\sb(p,q)$ if and only if one of
  the following conditions holds, see \cite[Proposition 2.2]{NO09}:
  \begin{enumerate}[(a)]
  \item \label{type:a} $\pi^{-1}\sigma=\epsilon_i$ and the cycle containing $i$
    in $\pi$ is nonzero, i.e., $\pi$ has $((i \cdots))$ and $\sigma$ has $[i
    \cdots]$.
  \item \label{type:b} $\pi^{-1} \sigma=((i~j))$ and no two of $i$, $-i$, $j$,
    $-j$ belong to the same cycle in $\pi$ with $\abs{i} \neq \abs{j}$, i.e.,
    $\pi$ has $((i \cdots))((j \cdots))$ and $\sigma$ has $((i \cdots j
    \cdots))$.
  \item \label{type:c} $\pi^{-1} \sigma=((i~j))$ and the cycle containing $i$ in
    $\pi$ is nonzero and the cycle containing $j$ in $\pi$ is zero with $\abs{i}
    \neq \abs{j}$, i.e., $\pi$ has $((i \cdots))[j \cdots]$ and $\sigma$ has $[i
    \cdots -j \cdots]$.
  \item \label{type:d} $\pi^{-1} \sigma=((i~j))$ and $i$ and $-j$ belong to the
    same nonzero cycle in $\pi$ with $\abs{i} \neq \abs{j}$, i.e., $\pi$ has
    $((i \cdots {-j} \cdots))$ and $\sigma$ has $[i \cdots][{-j} \cdots]$.
\end{enumerate}
If $\sigma$ covers $\pi$ in $\sb(p,q)$, we have $\zc(\sigma)\ge\zc(\pi)$, where
$\zc(\sigma)$ is the the number of zero cycles in $\sigma$.  More precisely we
have
\[
\zc(\sigma)-\zc(\pi)=
\begin{cases}
0 & \text{if type \eqref{type:b} or \eqref{type:c},} \\
1 & \text{if type \eqref{type:a},}\\
2 & \text{if type \eqref{type:d}.}
\end{cases}
\]
Since $\gamma_{p,q}$ has two zero cycles, each $\pi \in \sb(p,q)$ has at most
two zero cycles. Moreover, if $\pi$ has two zero cycles, then one of them
belongs to $\set{\pm, 1, \dots, {\pm p}}$ and the other belongs to
$\set{\pm{(p+1)}, \dots, \pm{(p+q)}}$.  Consider a maximal chain $C$ in
$\sb(p,q)$.
\begin{itemize}
\item If $C$ has a permutation $\pi$ with $\zc(\pi)=1$, there are two cover
  relations of type \eqref{type:a} and no cover relations of type \eqref{type:d}
  in $C$. For each cover relation $\pi < \sigma$ of type \eqref{type:a},
  \eqref{type:b}, or \eqref{type:c}, $\sigma$ is obtained by merging cycles in
  $\pi$. Since $\gamma_{p,q}$ has only disconnected cycles, all permutations in
  $C$ are disconnected, which implies that $C$ is disconnected.
\item Otherwise, there is a cover relation $\pi < \sigma$ of type \eqref{type:d}
  in $C$. Then $\sigma$ has two zero cycles $[i \cdots]$ and $[{-j} \cdots]$,
  one of which is contained in $\set{\pm, 1, \dots, {\pm p}}$ and the other is
  contained in $\set{\pm{(p+1)}, \dots, \pm{(p+q)}}$. Thus $\pi$ has a connected
  nonzero cycle $((i \cdots {-j} \cdots))$, and $C$ is connected. Since $C$ has
  no cover relations of type \eqref{type:a}, $\varphi(C)$ has no transposition
  of the form $\epsilon_i$.
\end{itemize}

Therefore, if $C$ is a disconnected maximal chain of $\sb(p,q)$, then
$\varphi(C)$ has two transpositions of the form $\epsilon_i$. So all
transpositions of $\varphi(C)$ are disconnected.  Also, if $C$ is a connected
maximal chain of $\sb(p,q)$, then $\varphi(C)$ has no transposition of the form
$\epsilon_i$ and has at least one connected transposition.
\end{proof}

The following proposition is a refinement of \eqref{eq:nc-formula}.

\begin{prop} \label{prop:connected}
The number of disconnected maximal chains of $\sb(p,q)$ is equal to
\begin{align}
{p+q \choose q} p^p q^q\label{eq:dmc}
\end{align}
and the number of connected maximal chains of $\sb(p,q)$ is equal to
\begin{align}
\sum_{c \ge 1} 2c {p+q \choose p-c} p^{p-c} q^{q+c}.\label{eq:cmc}
\end{align}
\end{prop}

\begin{proof}
  Let $NC^{(B)}(n)$ denote the poset of \emph{noncrossing partitions of type $B$}
  of size $n$.  It is well-known \cite[Proposition 7]{Rei97} that the number of
  maximal chains of $NC^{(B)}(n)$ equals $n^n$.  Let $\gamma_p := [1~2~\dots~p]$
  and $\gamma_q := [p+1~p+2~\dots~p+q]$.  Since $NC^{(B)}(p) \simeq [\epsilon,
  \gamma_{p}]$ and $NC^{(B)}(q) \simeq [\epsilon, \gamma_{q}]$, the numbers of
  maximal chains of $[\epsilon, \gamma_{p}]$ and $[\epsilon, \gamma_{q}]$ are
  respectively $p^p$ and $q^q$.  Given two maximal chains $C_1$ of $[\epsilon,
  \gamma_{p}]$ and $C_2$ of $[\epsilon, \gamma_{q}]$, we can obtain a
  $(p+q)$-tuple $(\tau_1,\dots,\tau_{p+q})$ of transpositions by shuffling
  $\varphi(C_1)$ and $\varphi(C_2)$ in ${p+q \choose q}$ ways. Then
  $C=\{\pi_0<\pi_1<\dots<\pi_{p+q}\}$, where $\pi_i=\tau_1\cdots\tau_i$, is a
  disconnected maximal chain of $\sb(p,q)$. By Lemma~\ref{lem:linked}, it is
  easy to see that every disconnected maximal chain of $\sb(p,q)$ can be
  obtained in this way. Thus we get \eqref{eq:dmc}.  By \eqref{eq:nc-formula}
  and \eqref{eq:dmc}, we obtain \eqref{eq:cmc}.
\end{proof}

\begin{rmk}
While one can also deduce Proposition~\ref{prop:connected} using the proof
of Theorem 5.3 in \cite{GNO11}, our proof gives a direct combinatorial
interpretation of \eqref{eq:dmc}.
\end{rmk}

We now prove the following identity that appears in the introduction.  The proof
is due to Krattenthaler \cite{Kra}.

\begin{lem}\label{lem:kk}
We have
\begin{align}
\sum_{c \ge 1} c {p+q \choose p-c} p^{p-c} q^{q+c}
=
\frac{pq}{p+q} {p+q \choose q} p^p q^q. \label{eq:kk-formula}
\end{align}
\end{lem}
\begin{proof}
Since $c = p\cdot \frac{q+c}{p+q} - q \cdot \frac{p-c}{p+q}$, we have
\begin{align*}
\sum_{c=0}^p c {p+q \choose p-c} p^{p-c} q^{q+c}
&= \sum_{c=0}^p  \left( p\cdot \frac{q+c}{p+q} - q \cdot \frac{p-c}{p+q} \right) {p+q \choose p-c} p^{p-c} q^{q+c}\\
&= \sum_{c=0}^p \left( {p+q-1 \choose p-c} p^{p-c+1} q^{q+c} - {p+q-1 \choose p-c-1} p^{p-c} q^{q+c+1}\right)\\
&= {p+q-1 \choose p} p^{p+1} q^{q} = \frac{pq}{p+q}{p+q \choose p} p^p q^q.
\end{align*}
\end{proof}

By Proposition~\ref{prop:connected} and Lemma~\ref{lem:kk}, we get the following.
\begin{cor} \label{cor:cmc}
The number of connected maximal chains of $\sb(p,q)$ is equal to
\begin{align}
\frac{2pq}{p+q} {p+q \choose q} p^p q^q. \label{eq:cmc-kk}
\end{align}
\end{cor}

For example, Figure~\ref{fig:connected} illustrates
$16=\frac{4}{3}\,\binom{3}{1}\,2^2$ connected maximal chains of $\sb(2,1)$.

By Corollary~\ref{cor:cmc}, in order to prove \eqref{eq:pq-formula}
combinatorially it is sufficient to find a 2-1 map from $\CM(\sb(p,q))$
to $\F_{(p,q)}$. We will find such a map in the next section.

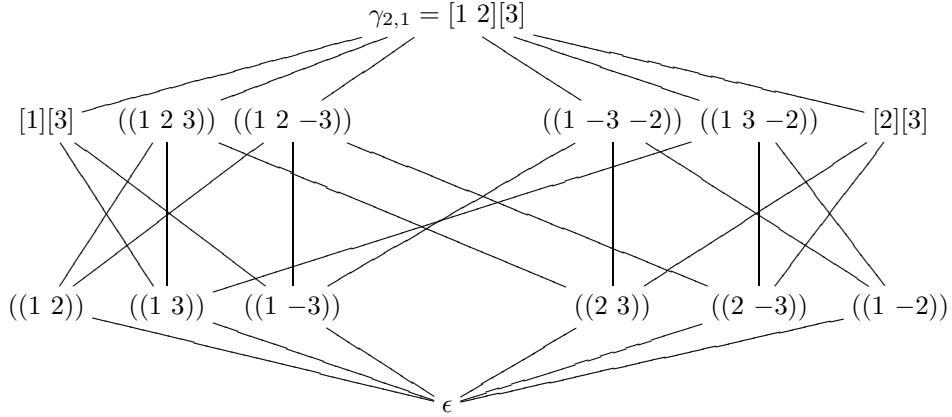
\begin{figure}[t]
$$
\xy
\small
\xymatrixcolsep{0px}
\xymatrix{
 &&&& \gamma_{2,1}=[1~2][3] \ar@{-}[lllld] \ar@{-}[lld] \ar@{-}[ld]
\ar@{-}[rd] \ar@{-}[rrd] \ar@{-}[rrrrd] &&&& \\
 [1][3] \ar@{-}[rrdd] \ar@{-}[rrrdd]
 & {}
 & ((1~2~3)) \ar@{-}[lldd] \ar@{-}[dd] \ar@{-}[rrrdd]
 & ((1~2~{-3})) \ar@{-}[llldd] \ar@{-}[dd] \ar@{-}[rrrdd]
 & {}
 & ((1~{-3}~{-2})) \ar@{-}[lldd] \ar@{-}[dd] \ar@{-}[rrrdd]
 & ((1~3~{-2})) \ar@{-}[lllldd] \ar@{-}[dd] \ar@{-}[rrdd]
 & {}
 & [2][3] \ar@{-}[llldd] \ar@{-}[lldd]
 \\
 &&&&&&&&\\
 ((1~2)) & {} & ((1~3)) & ((1~{-3})) & {} & ((2~3)) & ((2~{-3})) &
{} & ((1~{-2})) \\
 &&&& \epsilon \ar@{-}[llllu] \ar@{-}[llu] \ar@{-}[lu]  \ar@{-}[ru]
\ar@{-}[rru] \ar@{-}[rrrru] &&&&
}
\endxy
$$
\caption{Connected maximal chains in $\sb(2,1)$.}
\label{fig:connected}
\end{figure}

\begin{rmk}
  One can check that the factorizations $\varphi(C)$ coming from connected
  maximal chains $C$ in $\sb(p,q)$ are precisely the minimal factorizations of
  $\gamma_{p,q}$ in the Weyl group $D_n$. Thus Corollary~\ref{cor:cmc} can be
  restated as follows: the number of minimal factorizations of $\gamma_{p,q}$ in
  $D_n$ is equal to $\frac{2pq}{p+q} {p+q \choose q} p^p q^q$.  Goupil
  \cite[Theorem 3.1]{Gou95} also proved this result by finding a recurrence
  relation.
\end{rmk}

\begin{rmk}
  Since the proof of Lemma~\ref{lem:kk} is a simple manipulation, it is easy and
  straightforward to construct a combinatorial proof for the identity in
  Lemma~\ref{lem:kk}. Together with the result in Section~\ref{sec:2-1-map} we
  get a combinatorial proof of \eqref{eq:cmc-kk}.  It would be interesting to
  find a direct bijective proof of \eqref{eq:cmc-kk} without using
  Lemma~\ref{lem:kk}.
\end{rmk}

\section{A 2-1 map from $\CM(\sb(p,q))$ to $\F_{(p,q)}$}
\label{sec:2-1-map}

Recall that a minimal transitive factorization of $\alpha_{p,q}= (1 \dots p)(p+1
\dots p+q)$ is a sequence $(\eta_1,\ldots, \eta_{p+q})$ of transpositions in
$\S_{p+q}$ such that
\begin{enumerate}[(1)]
\item $\eta_1 \cdots \eta_{p+q} = \alpha_{p,q}$ and
\item $\set{\eta_1, \dots, \eta_{p+q}}$ generates $\S_{p+q}$,
\end{enumerate}
and $\F_{(p,q)}$ is the set of minimal transitive factorizations of $\alpha_{p,q}$.

In this section we will prove the following theorem.

\begin{thm}\label{thm:main}
  There is a 2-1 map from the set of connected maximal chains in $\sb(p,q)$ to
  the set $\F_{(p,q)}$ of minimal transitive factorizations of $\alpha_{p,q}$.
\end{thm}

In order to prove Theorem~\ref{thm:main} we need some definitions.

\begin{defn}[Two maps $(\cdot)^+$ and $\abs{\cdot}$]
  We introduce the following two maps.
\begin{enumerate}[(1)]
\item The map $(\cdot)^+: B_n \to B_n$ is defined by
$$
\sigma^+(i) =
\begin{cases}
\abs{\sigma(i)} &\text{ if~ $i>0$,}\\
-\abs{\sigma(i)} &\text{ if~ $i<0$.}
\end{cases}
$$
\item The map $\abs{\cdot}: B_n \to \S_n$ is defined by $\abs{\sigma}(i) = \abs{\sigma(i)}$ for all $i \in \set{1,\ldots,n}$.
\end{enumerate}
\end{defn}

\begin{defn}
A $(p+q)$-tuple $(\tau_1,\ldots, \tau_{p+q})$ of transpositions in $B_{p+q}$ is called a \emph{minimal transitive factorization of type $B$} of $\gamma_{p,q}=[1 \dots p][p+1 \dots  p+q]$ if it satisfies
\begin{enumerate}[(1)]
\item $\tau_1 \dots \tau_{p+q} = \gamma_{p,q}$,
\item $\set{\abs{\tau_1}, \dots, \abs{\tau_{p+q}}}$ generates $\S_{p+q}$.
\end{enumerate}
Denote by $\F_{(p,q)}^{(B)}$ the set of minimal transitive factorizations of type $B$ of $\gamma_{p,q}$.
\end{defn}

\begin{defn}
A $(p+q)$-tuple $(\sigma_1,\ldots, \sigma_{p+q})$ of transpositions in $B_{p+q}$ is called a \emph{positive} minimal transitive factorization of type $B$ of $\beta_{p,q}= ((1 \dots p))((p+1 \linebreak[2] \dots  p+q))$ if it satisfies
\begin{enumerate}[(1)]
\item $\sigma_1 \dots \sigma_{p+q} = \beta_{p,q}$,
\item $\set{\abs{\sigma_1}, \dots, \abs{\sigma_{p+q}}}$ generates $\S_{p+q}$,
\item $\sigma_i=\sigma_i^+$ for all $i=1,\ldots,p+q$.
\end{enumerate}
Denote by $\F_{(p,q)}^{+}$ the set of positive minimal transitive factorizations of type $B$ of $\beta_{p,q}$.
\end{defn}

For the rest of this section we will prove the following:
\begin{enumerate}
\item The map $\varphi:\CM(\sb(p,q)) \to \F_{(p,q)}^{(B)}$ is a
  bijection. (Lemma~\ref{lem:varphi})
\item There is a 2-1 map $\left( \cdot \right)^+:\F_{(p,q)}^{(B)} \to
  \F_{(p,q)}^{+}$. (Lemma~\ref{lem:plus})
\item There is a bijection $\left| \cdot \right |:\F_{(p,q)}^{+} \to
  \F_{(p,q)}$. (Lemma~\ref{lem:abs})
\end{enumerate}
By the above three statements the composition $\left| \varphi ^+ \right|:= \left| \cdot \right | \circ \left( \cdot \right)^+ \circ \varphi$ is a
2-1 map from $\CM(\sb(p,q))$ to $\F_{(p,q)}$, which completes the proof of
Theorem~\ref{thm:main}.  Since the proofs of the first and the third statements
are simpler, we will present these first.


\begin{lem}\label{lem:varphi}
  The map $\varphi:\CM(\sb(p,q)) \to \F_{(p,q)}^{(B)}$ is a bijection.
\end{lem}

\begin{proof}
  Given a connected maximal chain $C=\{\epsilon = \pi_0 < \pi_1< \dots <
  \pi_{p+q} =\gamma_{p,q}\}$ in $\sb(p,q)$, the elements in the sequence
  $\varphi(C)=(\tau_{1},\dots,\tau_{p+q})$ are transpositions with
  $\tau_1\cdots\tau_{p+q}=\gamma_{p+q}$.  By Lemma~\ref{lem:linked}, at least
  one of $\tau_i$'s is connected. Thus $\set{\abs{\tau_1}, \dots,
    \abs{\tau_{p+q}}}$ generates $\S_{p+q}$, and $\varphi(C)\in
  \F_{(p,q)}^{(B)}$.  Conversely, if $\tau=(\tau_{1},\dots,\tau_{p+q})\in
  \F_{(p,q)}^{(B)}$, then $\varphi^{-1}(\tau)=\{\epsilon = \pi_0 < \pi_1< \dots
  < \pi_{p+q} =\gamma_{p,q}\}$, where $\pi_i=\tau_1\cdots\tau_i$, is a connected
  maximal chain in $\sb(p,q)$ because $\{|\tau_1|,\dots,|\tau_{p+q}|\}$ generates
  $\S_{p+q}$.
\end{proof}

\begin{lem}\label{lem:abs}
There is a bijection $\left| \cdot \right |:\F_{(p,q)}^{+} \to \F_{(p,q)}$.
\end{lem}

\begin{proof}
  Let $(\sigma_1, \dots, \sigma_{p+q})\in \F_{(p,q)}^{+}$. Each $\sigma_i$ can
  be written as $\sigma_i = ((j~k))$ for some positive integers $j$ and $k$. In
  this case we let $\eta_i=\left| \sigma_i \right| = (j~k)\in S_{p+q}$. Then the
  map $\left| \cdot \right |:\F_{(p,q)}^{+} \to \F_{(p,q)}$ sending $(\sigma_1,
  \dots, \sigma_{p+q})$ to $(\eta_1, \dots, \eta_{p+q})$ is a bijection.
\end{proof}

Recall $\epsilon_i = [i] =(i~-i)$.  We write $\overline{((i~j))} :=
((i~{-j}))$.  It is easy to see that for $i, j \in \set{\pm 1, \dots, \pm
  (p+q)}$, we have
\begin{align}
[i~j] = \epsilon_i ((i ~j)) = ((i~j)) \epsilon_j = \overline{((i~j))} \epsilon_i= \epsilon_j \overline{((i~j))}.
\label{eq:relation}
\end{align}

\begin{lem}\label{lem:plus}
  There is a 2-1 map $\left( \cdot \right)^+:\F_{(p,q)}^{(B)} \to \F_{(p,q)}^{+}$.
\end{lem}

\begin{proof}
  For $(\tau_{1},\tau_{2},\dots,\tau_{p+q})\in \F_{(p,q)}^{(B)}$, we define
  $\left(\tau_{1},\tau_{2},\dots,\tau_{p+q}\right)^+ =
  (\tau_{1}^+,\tau_{2}^+,\dots,\tau_{p+q}^+)$.  Since
  $\tau_{1}^+\dots\tau_{p+q}^+=\gamma_{p,q}^+=\beta_{p,q}$, we have
  $\left(\tau_{1},\tau_{2},\dots,\tau_{p+q} \right)^+\in \F_{(p,q)}^{+}$.

The map $\left( \cdot \right)^+$ is surjective: Suppose
$(\sigma_{1},\sigma_{2},\dots,\sigma_{p+q}) \in \F_{(p,q)}^{+}$. Since
$\sigma_{1}\sigma_{2}\dots\sigma_{p+q}=\beta_{p,q}$ and $\gamma_{p,q} =
\epsilon_{p+1} \epsilon_1 \beta_{p,q}$, we have
\begin{align}
  \gamma_{p,q} = \epsilon_{p+1}
  \epsilon_1\sigma_{1}\sigma_{2}\dots\sigma_{p+q}. \label{eq:beta_gamma}
\end{align}
By \eqref{eq:relation}, if $\sigma_{\ell} = ((u~v))$, we have
\begin{align}
  \epsilon_u \sigma_{1}\sigma_{2}\dots\sigma_{p+q} = \epsilon_v
  \widetilde{\sigma}_{1}\dots\widetilde{\sigma}_{\ell-1}
  \overline{\sigma_{\ell}}\sigma_{\ell+1} \cdots \sigma_{p+q},
\label{eq:convert}
\end{align}
where
$$\widetilde{\sigma}_i =
\begin{cases}
\overline{\sigma_i} &\text{if $\sigma_i$ has either $u$ or $v$,}\\
\sigma_i &\text{otherwise.}
\end{cases}
$$
Since $(\sigma_{1},\sigma_{2},\dots,\sigma_{p+q})$ is transitive, we can find
integers $1 = a_0, a_1, \dots, a_k = p+1$ such that $((a_{i-1}, a_i))\in
\set{\sigma_{1},\sigma_{2},\dots,\sigma_{p+q}}$ for all $1\leq i \leq k$.  Using
this fact and the relation in \eqref{eq:convert}, we can rewrite
\eqref{eq:beta_gamma} as
\[
\gamma_{p,q}
= \epsilon_{p+1} (\epsilon_{p+1} \tau_{1}\tau_{2}\dots\tau_{p+q})
= \tau_{1}\tau_{2}\dots\tau_{p+q},
\]
where $\tau_i = \sigma_i$ or $\overline{\sigma_i}$ for all $i=1,2,\dots, p+q$.
Hence, $\left( \tau_{1},\tau_{2},\dots,\tau_{p+q} \right)^+ =
(\sigma_{1},\sigma_{2},\dots,\sigma_{p+q})$ and $\left( \cdot \right)^+$ is surjective.

We need to show that $\left( \cdot \right)^+$ is two-to-one.  Let us fix
$\sigma=(\sigma_{1},\sigma_{2},\dots,\sigma_{p+q}) \in \F_{(p,q)}^{+}$. Since
$\left( \cdot \right)^+$ is surjective, there is $\tau=(\tau_1, \dots, \tau_{p+q})\in
\F_{(p,q)}^{(B)}$ satisfying $\tau^+=\sigma$. Then $\tau'=(\tau'_1 \dots
\tau'_{p+q}) \in \F_{(p,q)}^{(B)}$ defined by
\begin{equation}
  \label{eq:tau'}
\tau'_i =
\begin{cases}
\tau_i &\text{if $\tau_i$ is disconnected}\\
\overline{\tau_i} &\text{if $\tau_i$ is connected,}
\end{cases}
\end{equation}
also satisfies $\left( \tau' \right)^+=\sigma$.  Since $\tau$ has at least one connected
transposition, $\tau\ne\tau'$. Hence the preimage of $\sigma$ under $\left( \cdot \right)^+$ has at least two
elements. In order to prove that the preimage of $\sigma$ under $\left( \cdot \right)^+$ has exactly two elements,
we consider the sets
\begin{align*}
  \F&=\set{(\tau_1,\ldots,\tau_{p+q}): \tau_i^+
    = \sigma_i \text{ for all $i=1,\ldots,p+q$}},\\
  \F(\delta) &=\set{(\tau_1,\ldots,\tau_{p+q}):
    \tau_1\dots\tau_{p+q}=\delta,~ \tau_i^+ = \sigma_i \text{ for all
      $i=1,\ldots,p+q$}}.
\end{align*}
Then $\F$ has $2^{p+q}$ elements, and all preimages of $\sigma$ under $\left( \cdot \right)^+$  belong to $\F(\gamma_{p,q})$.

Suppose $(\tau_1,\ldots,\tau_{p+q}) \in \F(\delta)$.  Since $\delta^+ =
\beta_{p,q}$, we have $(\delta~{\beta_{p,q}}^{-1})^{+}=\epsilon$. Since
$\delta~{\beta_{p,q}}^{-1}$ is an even permutation as a permutation on $\set{\pm
  1,\ldots,\pm (p+q)}$, we have $\delta~{\beta_{p,q}}^{-1} = \epsilon_{i_1}
\cdots \epsilon_{i_{2k}}$ for some $1\le i_1 < \cdots < i_{2k} \le p+q$.  Thus
the set
$$\ind(\delta):=\set{i: \delta(i)=-\beta_{p,q}(i)\text{ and } 1\le i \le p+q}$$
has even cardinality.  Define the set
$$B(\beta_{p,q}) = \set{\delta\in B_{p+q}: \text{$\delta^+ = \beta_{p,q}$ and $\# \ind(\delta)$ is even}},$$
whose cardinality is $2^{p+q-1}$. Then we have
\begin{equation}
  \label{eq:F}
\#\F = \sum_{\delta\in B(\beta_{p,q})} \#\F(\delta).
\end{equation}

We claim that $\#\F(\delta)\ge2$ for each $\delta\in B(\beta_{p,q})$. Then by
the claim together with $\#\F=2^{p+q}$, $\#\F(\delta)=2^{p+q-1}$, and
\eqref{eq:F}, we get $\#\F(\delta)=2$ for each $\delta\in B(\beta_{p,q})$. In
particular, we have $\#\F(\gamma_{p,q})=2$, which implies that
the preimage of $\sigma$ under $\left( \cdot \right)^+$ has exactly two elements, thus completing the proof of this
lemma.

It remains to show the claim. Suppose $\delta \in B(\beta_{p,q})$. Then we have
\begin{align}
\delta = \left( \prod_{i \in \ind(\delta)} \epsilon_{i} \right) \sigma_{1}\sigma_{2}\dots\sigma_{p+q}. \label{eq:delta}
\end{align}
Using the relation \eqref{eq:convert} and the transitivity, we can rewrite
\eqref{eq:delta} as
\[
\delta = \epsilon_{1}^{\#\ind(\delta)} \tau_{1}\tau_{2}\dots\tau_{p+q}
= \tau_{1}\tau_{2}\dots\tau_{p+q},
\]
where $\tau_i = \sigma_i$ or $\overline{\sigma_i}$ for all $i=1,2,\dots, p+q$.
Then $\F(\delta)$ has at least two elements $(\tau_1, \dots, \tau_n)$ and
$(\tau'_1, \dots, \tau'_n)$, the latter is defined by \eqref{eq:tau'}. Thus
$\#\F(\delta)\ge2$ and we are done.
\end{proof}

For example, let $\sigma = \left(~((1~2)),((2~5)),((2~3)),((4~5)),((3~4))~
\right)\in \F_{(3,2)}^+$ be the following factorization
\[
\beta_{3,2} = ((1~2~3))((4~5)) = ((1~2))~((2~5))~((2~3))~((4~5))~((3~4)).
\]
Since $\gamma_{3,2} = \epsilon_4 ~ \epsilon_1 ~ \beta_{3,2}$, we can obtain a
factorization of $\gamma_{3,2}$ from $\sigma$ as follows:
\begin{align*}
\gamma_{3,2} = [1~2~3][4~5]
& = \epsilon_4 ~ \epsilon_1~((1~2))~((2~5))~((2~3))~((4~5))~((3~4))\\
& = \epsilon_4 ~ \epsilon_2~\overline{((1~2))}~((2~5))~((2~3))~((4~5))~((3~4))\\
& = \epsilon_4 ~ \epsilon_3~((1~2))~\overline{((2~5))}~\overline{((2~3))}~((4~5))~((3~4))\\
& = \epsilon_4 ~ \epsilon_4~((1~2))~\overline{((2~5))}~((2~3))~\overline{((4~5))}~\overline{((3~4))}\\
& = ((1~2))~\overline{((2~5))}~((2~3))~\overline{((4~5))}~\overline{((3~4))}.
\end{align*}
Thus $\tau = \left(~
  ((1~2)),\overline{((2~5))},((2~3)),\overline{((4~5))},\overline{((3~4))} ~
\right)\in \F_{(3,2)}^{(B)}$ satisfies $\tau^+ = \sigma$. The factorization
$\tau' = \left(~((1~2)),((2~5)),((2~3)),\overline{((4~5))},((3~4))~ \right)$
obtained by toggling the connected transpositions of $\tau$ also satisfies
$\left( \tau' \right)^+ = \sigma$.

\section{Marked annular noncrossing permutations of type $A$}
\label{sec:mark-annul-noncr}

Mingo and Nica \cite{MN04} studied the set $\sa(p,q)=\{\pi\in \S_n: \pi\leq
\alpha_{p,q}\}$ of annular noncrossing permutations of type $A$. In contrast to
the type $B$ case, $\sa(p,q)$ is not isomorphic to the set $NC^A(p,q)$ of
annular noncrossing partitions of type $A$. In fact, the two sets
$\sa(p,q)$ and $NC^A(p,q)$ have different cardinalities, see \cite[Section~4]{MN04}.

In what follows we construct a poset whose maximal chains are in bijection with
minimal transitive factorizations of $\alpha_{p,q}$.

Recall that the \emph{(absolute) length} $\ell(\pi)$ of $\pi\in \S_n$ is defined
to be the smallest integer $k$ such that $\pi$ can be written as a product of
$k$ transpositions. Equivalently, $\ell(\pi) = n - \cycle(\pi)$, where
$\cycle(\pi)$ is the number of cycles in $\pi$. The \emph{(absolute) order}
$\pi\leq \sigma$ is defined if and only if $\ell(\sigma) = \ell(\pi) +
\ell(\pi^{-1}\sigma)$.  In this order, the interval $[\epsilon, \alpha_{p,q}]$
is isomorphic to $[\epsilon,(1,2,\dots,p)]\times [\epsilon,
(p+1,p+2,\dots,p+q)]$.

Similarly to the type $B$ case, we say that $\pi\in \S_{p+q}$ is
\emph{connected} if $\pi$ has a cycle intersecting with both $\{1,2,\dots,p\}$
and $\{p+1,p+2,\dots,p+q\}$, and \emph{disconnected} otherwise.

A \emph{marked annular noncrossing permutation of type $A$} is a pair $(\pi,z)$
of a permutation $\pi\in\S_{p+q}$ and an integer $z\in\{0,1\}$ such that
\begin{enumerate}
\item if $\pi$ is disconnected, then $\pi\leq \alpha_{p,q}$, i.e.
  $\ell(\alpha_{p,q}) = \ell(\pi) + \ell(\pi^{-1}\alpha_{p,q})$,
\item if $\pi$ is connected, then $\ell(\alpha_{p,q}) = \ell(\pi) +
  \ell(\pi^{-1}\alpha_{p,q})-2$,
\item if $z=1$, then $\pi$ is disconnected.
\end{enumerate}

We denote by $\msa(p,q)$ the set of marked annular noncrossing permutations of
type $A$. We define the partial order $\leq$ on $\msa(p,q)$ as
follows: $(\pi,z)\leq (\sigma,w)$ if and only if one of the
following holds:
\begin{enumerate}
\item $z=w$ and $\pi\leq\sigma$,
\item $z=0, w=1$, $\pi$ is connected, $\sigma$ is disconnected, and
  $\ell(\sigma) = \ell(\pi) + \ell(\pi^{-1}\sigma)-2$.
\end{enumerate}


Then $\msa(p,q)$ is a graded poset of rank $p+q$ with minimum $\hat0=(\epsilon,
0)$ and maximum $\hat1=(\alpha_{p,q},1)$. The rank function of $\msa(p,q)$ is
given by
\[
\rank(\pi,z)=\ell(\pi)+2z.
\]

We say that a multichain $(\pi_1,z_1)\leq(\pi_2,z_2)\leq\cdots\leq(\pi_{m},z_m)$
of $\msa(p,q)$ is \emph{connected} if it contains at least one connected
permutation, and \emph{disconnected} otherwise.

We now show the relation between the maximal chains of $\msa(p,q)$ and the
minimal transitive factorizations of $\alpha_{p,q}$.

\begin{prop}\label{prop:mtf}
  There is a bijection between the set of maximal chains of $\msa(p,q)$ and the
  set of minimal transitive factorizations of $\alpha_{p,q}$. Moreover, every
  maximal chain of $\msa(p,q)$ is connected.
\end{prop}
\begin{proof}
  Let $(\epsilon,0) = (\pi_0,z_0) < (\pi_1,z_1) < \cdots < (\pi_{p+q},z_{p+q}) =
  (\alpha_{p,q},1)$ be a maximal chain in $\msa(p,q)$. By definition of the
  partial order on $\msa(p,q)$ there is a unique integer $k$ such that
  $z_0=z_1=\cdots = z_{k-1} = 0$ and $z_{k}=z_{k+1}=\cdots=z_{p+q}=1$.

  Suppose $i\in\{1,2,\dots,p+q\}\setminus\{k\}$.  We have
  $\ell(\pi_{i})=\ell(\pi_{i-1}) +\ell(\pi_{i-1}^{-1}\pi_{i})$. Since
  $z_i=z_{i-1}$ and
\[
\ell(\pi_{i})+2z_i=\rank(\pi_{i},z_{i})=\rank(\pi_{i-1},z_{i-1})+1=\ell(\pi_{i-1})+2z_{i-1}+1,
\]
we get $\ell(\pi_{i}) = \ell(\pi_{i-1})+1$. Thus $\ell(\pi_{i-1}^{-1}\pi_{i})=1$
and $t_i = \pi_{i-1}^{-1}\pi_{i}$ is a transposition. Furthermore, since $\pi_k,
\pi_{k+1},\dots,\pi_{p+q}$ are disconnected, so are
$t_{k+1},t_{k+2},\dots,t_{p+q}$.


On the other hand, we have $\ell(\pi_{k}) = \ell(\pi_{k-1}) +
\ell(\pi_{k-1}^{-1}\pi_{k}) - 2$ and $\ell(\pi_{k}) +2=\rank(\pi_{k},z_{k}) =
\rank(\pi_{k-1},z_{k-1})+1 = \ell(\pi_{k-1})+1$.  Thus
$\ell(\pi_k^{-1}\pi_{k+1})=1$ and $t_k = \pi_k^{-1}\pi_{k+1}$ is a connected
transposition. In particular, $t_k$ is the last connected transposition in
$t_1,t_2,\dots,t_{p+q}$. It is easy to see that the map sending the maximal
chain to $(t_1,t_2,\dots,t_{p+q})$ is a desired bijection.
\end{proof}




\begin{prop}\cite[Proposition~5.1]{GNO11}\label{prop:B}
  For positive integers $p,q,m$, there is a bijection between the set of tuples
  $(c,d; L^E, R_1^E,\dots,R_m^E; L^I,R_1^I,\dots,R_m^I)$ satisfying $c\geq1$,
  $1\le d\le 2c$, and
  \begin{equation}
    \label{eq:cond1}
 L^E, R_1^E,\dots,R_m^E\subseteq\{1,2,\dots,p\},\quad |L^E|=|R_1^E|+\cdots+|R_m^E|+c,
  \end{equation}
  \begin{equation}
    \label{eq:cond2}
    L^I, R_1^I,\dots,R_m^I\subseteq\{p+1,p+2,\dots,p+q\},\quad
    |L^I|=|R_1^I|+\cdots+|R_m^I|-c,
  \end{equation}
  and the set of connected multichains $\pi_1\leq\pi_2\leq\cdots\leq\pi_{m}$ in
  $\sb(p,q)$ such that
\[
\rank(\pi_i) = p+q- \left(
|R_i^E|+\cdots+|R_m^E|+|R_i^I|+\cdots+|R_m^I|
\right), \quad 1\leq i\leq m.
\]
\end{prop}

We now prove a type $A$ analog of Proposition~\ref{prop:B}. Our proof for the
type $A$ analog is almost the same as the proof of Proposition~5.1 in
\cite{GNO11}. The only difference (except the obvious difference caused by sign)
is that for the type $A$ case we have to determine the first and the last
elements of a cycle.  More precisely, for a cycle of the form
$((a_1,\dots,a_k))$ in $\sb(p,q)$ whose elements are contained in
$\{\pm1,\pm2,\dots,\pm p\}$, there is a unique way to write the cycle as
$((b_1,\dots,b_k))$ such that $b_1,\dots,b_k,-b_1,\dots,-b_k$ are in the same
cyclic order as the subsequence $1,2,\dots,p,-1,-2,\dots,-p$ consisting of $\pm
a_1,\pm a_2,\dots,\pm a_k$. Thus we can naturally say that $b_1$ (or $-b_1$) is
the first element and $b_k$ (or $-b_k$) is the last element of the cycle. For
instance, consider the cycle $((1,2,-4))$ in, say, $\sb(4,3)$. Then $((-4,1,2))$
is the only way so that $-4,1,2,4,-1,-2$ are in the same cyclic order as the
subsequence $1,2,4,-1,-2,-4$ of $1,2,3,4,-1,-2,-3,-4$. If we write the cycle as
$((1,2,-4))$, the sequence $1,2,-4,-1,-2,4$ is not in the same cyclic order as
$1,2,4,-1,-2,-4$. However, for the cycle $(1,2,4)$ in $\msa(4,3)$, all three
cyclic shifts of $1,2,4$ are, of course, in the same cyclic order as the
subsequence $1,2,4$ of $1,2,3,4$.

\begin{prop}\label{prop:A}
  For positive integers $p$, $q$, and $m$, there is a bijection between the set of tuples
  $(c,d; L^E, R_1^E,\dots,R_m^E; L^I,R_1^I,\dots,R_m^I)$ satisfying $c\geq1$,
  $1\le d\le c$, \eqref{eq:cond1}, and \eqref{eq:cond2}, and the set of
  connected multichains $(\pi_1,z_1)\leq(\pi_2,z_2)\leq\cdots\leq(\pi_{m},z_m)$
  in $\msa(p,q)$ such that
  \begin{equation}
    \label{eq:rankA}
\rank(\pi_i,z_i) = p+q- \left(
|R_i^E|+\cdots+|R_m^E|+|R_i^I|+\cdots+|R_m^I|
\right), \quad 1\leq i\leq m.
  \end{equation}
\end{prop}
\begin{proof}
  Consider a connected multichain
  $(\pi_1,z_1)\leq(\pi_2,z_2)\leq\cdots\leq(\pi_{m},z_m)$ in $\msa(p,q)$. We
  will define the corresponding tuple $(c,d; L^E, R_1^E,\dots,R_m^E;
  L^I,R_1^I,\dots,R_m^I)$ according to the following steps.

  \emph{Step 1.} We first determine two integers $a\in\{1,2,\dots,p\}$ and
  $b\in\{p+1,\dots,p+q\}$.  Let $k$ be the largest index such that $\pi_k$ is
  connected. Then $\pi_k$ has one or more connected cycles. Take the connected
  cycle $C_{\max}$ of $\pi_k$ with largest element. Since $C_{\max}$ is
  connected, it can be uniquely written as $C_{\max}=(a_1,\dots, a_r,
  b_1,\dots,b_s)$, where $1\le a_1,\dots,a_r \le p$ and $p+1\le b_1,\dots,b_s\le
  p+q$. We let $a=a_1$ and $b=b_s$.

  \emph{Step 2.} To each $(\pi_i,z_i)$ we associate a tuple $(L_i^E,R_i^E;
  L_i^I, R_i^I)$ as follows. First, we set $L_i^E=R_i^E=L_i^I=R_i^I=\emptyset$,
  and for every cycle $C$ of $(\pi_i,z_i)$ we do the following. If $C$ is
  contained in $\{1,2,\dots,p\}$, add to $L_i^E$ (resp.~$R_i^E$) the element of
  $C$ that appears first (resp.~last) in the sequence
  $a,a+1,\dots,p,1,2,\dots,a-1$. If $C$ is contained in $\{p+1,p+2,\dots,p+q\}$,
  add to $L_i^I$ (resp.~$R_i^I$) the element of $C$ that appears first
  (resp.~last) in the sequence $b+1,b+2,\dots,p+q,p+1,p+2,\dots,b$. If $C$ is a
  connected cycle, it can be uniquely written as $C=(g_1,\dots, g_u,
  h_1,\dots,h_v)$, where $1\le g_1,\dots,g_u \le p$ and $p+1\le h_1,\dots,h_v\le
  p+q$. In this case we add $g_1$ to $L_i^E$ and $h_v$ to $R_i^I$.

  \emph{Step 3.} Let $L^E = L_1^E\cup \cdots \cup L_m^E$ and $L^I = L_1^I\cup
  \cdots \cup L_m^I$. Now consider the sequence
  \begin{equation}
    \label{eq:seq}
a,a+1,\dots,p,1,2,\dots,a-1, b+1,b+2,\dots,p+q,p+1,p+2,\dots,b
  \end{equation}
  with parenthesization obtained by placing a left parenthesis before every
  integer in $L^E\cup L^I$, a right parenthesis labeled $i$ after every integer
  in $R_i^E\cup R_i^I$ for all $i=1,2,\dots,m$. There may be more than one right
  parenthesis after one integer. In this case the right parentheses are placed
  in the increasing order of their labels. By the construction it is clear that
  the parenthesization is balanced. Now remove the integers larger than $p$ and
  their left and right parentheses in \eqref{eq:seq}. Then we have more left
  parentheses than right parentheses. Let $c$ be the number of left parentheses
  minus the number of right parentheses. Then there are exactly $c$ unmatched
  left parentheses. Let $j_1<j_2<\dots<j_c$ be the integers whose left
  parentheses are unmatched. Note that the left parenthesis of $a$ is unmatched
  because it was matched with a right parenthesis of $b$ before removing the
  numbers with parentheses. We define $d$ to be the index with $j_d=a$. We
  clearly have $1\le d\le c$.

  Then the map sending the multichain to $(c,d; L^E, R_1^E,\dots,R_m^E;
  L^I,R_1^I,\dots,R_m^I)$ is a desired bijection. The inverse map can be
  obtained in the same way as in the proof Proposition~5.1 in \cite{GNO11}.
\end{proof}

\begin{cor}
  The number of maximal chains in $\msa(p,q)$ is equal to
\[
\sum_{c \ge 1} c\binom{p+q}{p-c} p^{p-c} q^{q+c}.
\]
\end{cor}
\begin{proof}
  Let $M= \{(\pi_0,z_0) < (\pi_1,z_1) < \cdots < (\pi_{p+q},z_{p+q})\}$ be a
  maximal chain in $\msa(p,q)$.  By Proposition~\ref{prop:mtf}, $M$ is
  connected.  Let
  \begin{equation}
    \label{eq:5}
(c,d; L^E, R_0^E, R_1^E,\dots,R_{p+q}^E; L^I,R_0^I,R_1^I,\dots,R_{p+q}^I)
  \end{equation}
be the tuple corresponding to $M$ under the map in Proposition~\ref{prop:A}.
Since $M$ is maximal, using the conditions~\eqref{eq:cond1}, \eqref{eq:cond2},
and \eqref{eq:rankA}, we have
\[
|R_i^E|+|R_i^I|=1,  \quad i=0,1,\dots,p+q-1,
\]
\[
R_{p+q}^E=R_{p+q}^I=\emptyset, \quad
L^E=\{1,2,\dots,p\}, \quad L^I=\{p+1,p+2,\dots,p+q\},
\]
\[
|R_0^E|+\cdots+|R_{p+q-1}^E| = p-c, \quad
|R_0^I|+\cdots+|R_{p+q-1}^I| = q+c,
\]
It is now easy to see that for fixed $c$, the number of choices for the tuples
\eqref{eq:5} satisfying the above conditions is $\binom{p+q}{p-c} p^{p-c}
q^{q+c}$.
\end{proof}

\begin{rmk}
  Using Propositions~\ref{prop:B} and \ref{prop:A} we obtain a 2-1 map between
  the sets of connected multichains in $\sb(p,q)$ and in $\msa(p,q)$. It is not
  difficult to check that this 2-1 map is essentially the same as the 2-1 map in
  the previous section. Note also that by Propositions~\ref{prop:B} and
  \ref{prop:A} there is a 2-1 map from the set of connected permutations in
  $\sb(p,q)$ to the set of connected permutations in $\msa(p,q)$. However, we
  have $|\sb(p,q)|\ne 2|\msa(p,q)|$ because the numbers of disconnected
  permutations in $\sb(p,q)$ and in $\msa(p,q)$ are equal to
  $\binom{2p}{p}\binom{2q}{q}$ and $2 C_p C_q$ respectively, where $C_n =
  \frac{1}{n+1}\binom{2n}{n}$.
\end{rmk}

%
%
%

\section*{Acknowledgment}

We thank Christian Krattenthaler for the proof of Lemma~\ref{lem:kk}.
For to the second author, this research was supported by Basic Science Research Program through the National Research Foundation of Korea (NRF) funded by the Ministry of Education, Science
and Technology(2011-0008683). For the third author, this
work was supported by INHA UNIVERSITY Research Grant (INHA-44756).





\end{document}